\documentclass[11pt, english, reqno]{amsart}

\usepackage[T1]{fontenc}
\usepackage{babel}
\usepackage{mathrsfs}
\usepackage{amsthm}
\makeindex

\usepackage[dvips,top=3.8cm,left=4cm,right=3.8cm,
foot=3.8cm,bottom=4.0cm]{geometry}
\usepackage{amsfonts,amssymb,amsmath,amsthm, booktabs, 
latexsym}
\usepackage{centernot}
\usepackage{tikz-cd}
\usepackage{bbm}

\usepackage{enumerate}
\usepackage{color}

\newtheorem{theorem}{Theorem}[section]
\newtheorem{lemma}[theorem]{Lemma}
\newtheorem{proposition}[theorem]{Proposition}
\newtheorem{corollary}[theorem]{Corollary}
\theoremstyle{definition}
\newtheorem{definition}[theorem]{Definition}

\theoremstyle{remark}
\newtheorem{remark}[theorem]{Remark}

\numberwithin{equation}{section}

\numberwithin{equation}{section}
\newcommand{\be}{\begin{equation}}
\newcommand{\ee}{\end{equation}}
\newcommand{\ba}{\begin{aligned}}
\newcommand{\ea}{\end{aligned}}

\newcommand{\N}{{\mathbb N}}

\newcommand{\R}{{\mathbb R}}

\newcommand{\h}{{\mathcal H}}
\def\va{\varphi}
\def\la{\lambda}

\def\csi1{\circ\sigma^{-1}}

\def\ol{\overline}
\def\wt{\widetilde}
\def\wh{\widehat}

\def\mc{\mathcal}

\def\lc2{L^2_{\mathrm{loc}}(\mu)}
\newcommand{\B}{{\mathcal B}}
\newcommand{\Bfin}{\B_{\mathrm{fin}}}
\newcommand{\Dfin}{\mc D_{\mathrm{fin}}}
\newcommand{\Lloc}{L^1_{\mathrm{loc}}(\mu)}
\newcommand{\sms}{(V, \mathcal B, \mu)}
\newcommand{\VB}{(V, \mathcal B)}
\newcommand{\vv}{(V\times V, \B\times \B)}
\newcommand{\FVB}{{\mathcal F(V, \B)}}

\newcommand{\VtV}{V\times V}

\begin{document}

\title[Symmetric measures]{Markov operators generated by symmetric
 measures}

\author{Sergey Bezuglyi}
\address{Department of Mathematics, University of Iowa, Iowa City,
52242 IA, USA}
\email{sergii-bezuglyi@uiowa.edu}
\email{palle-jorgensen@uiowa.edu}

\author{Palle E.T. Jorgensen}




\begin{abstract} 
With view to applications, we here give an explicit correspondence
between the following two: (i) the set of symmetric and positive
measures $\rho$ on one hand, and (ii) a certain family of generalized
Markov transition measures $P$, with their associated Markov random 
walk models, on the other. By a generalized Markov transition measure we
mean a measurable and measure-valued function $P$ on $(V, \B)$, such 
that for every $x \in V , P(x;  \cdot)$ is a probability measure on 
$(V, \B$). Hence, with the use of our correspondence (i) - (ii), we study 
generalized Markov transitions $P$ and path-space dynamics. Given $P$,
 we introduce an associated operator, also denoted by $P$ , and we 
 analyze its spectral theoretic properties with reference to a system of 
 precise $L^2$ spaces.

Our setting is more general than that of earlier treatments of reversible 
Markov processes.
In a potential theoretic analysis of our processes, we introduce and study 
an associated energy Hilbert space $\h_E$, not directly linked to the initial 
$L^2$-spaces. Its properties are subtle, and our applications include a 
study of the $P$-harmonic functions. They may be in $\h_E$, called 
finite-energy harmonic functions. A second reason for $\h_E$ is that it 
plays a key role in our introduction of a generalized Greens function. (The 
latter stands in relation to our present measure theoretic Laplace operator 
in a way that parallels more traditional settings of Greens functions from 
classical potential theory.)
A third reason for $\h_E$  is its use in our analysis of path-space dynamics 
for generalized Markov transition systems.
\end{abstract}

\maketitle

\tableofcontents
\section{Introduction}\label{sect Intro}

In this paper, we continue our study of the graph Laplace and Markov
 operators, initiated in \cite{BezuglyiJorgensen2018}, 
which was based on the key notion of a $\sigma$-finite symmetric
 measure defined on the product space $\vv$ for a standard Borel 
 space $\VB$.  

Our goal is to extend the basic definitions and results of the theory of
weighted networks (known also as electrical or resistance networks)  to 
the case of measure spaces. We briefly recall that,  for a countable 
locally finite connected graph $G = (V, E)$
without loops, one can identify the edge set $E$ with a subset of the 
Cartesian product $\VtV$ and assign some weight $c_{xy}$ for every 
point $(x,y)$ in $E$ where $c_{xy}$ is a symmetric positive function. 
It gives us a symmetric atomic measure $\rho$ on $E$ whose projections
on $V$ are the counting measure $\mu$. Then, for a weighted network 
$(V, E, c)$, one defines the Markov transition probability kernel $P$ and
the graph Laplacian $\Delta = c(I -P)$ which are considered as operators 
acting either in $L^2$ spaces with respect to the measures $\mu$ and 
$\nu= c\mu$ or in the finite energy space $\h_E$. Their spectral 
properties are of great interest as well as the study of harmonic functions 
in the theory of weighted networks. 

Our approach to the measurable theory of weighted networks is based 
on the concept of a \textit{symmetric measure} defined on the 
Cartesian product $\vv$ where $(V, \B)$ is a standard Borel space.
(To stress the existing parallels we use the same notation as in discrete
case.) 
 In more detail, in the context of measurable dynamics,
the state space $V$ is considered very generally; more specifically 
$(V, \B)$ is given, where $\B$ is a specified $\sigma$-algebra for $V$.
 From $(V, \B)$, we 
 then form the corresponding product space, relative to the product 
 $\sigma$-algebra on $V \times V$. It is important that our initial measure 
 $\rho$  is not assumed finite, but only $\sigma$-finite. Since $\rho$  is 
 assumed symmetric, the respective two marginal measures coincide, here 
 denoted  $\mu$,  and they will also not be finite; only $\sigma$-finite. 
 The $\sigma$-finiteness will be a crucial fact in our computations of 
 a number  of Radon-Nikodym derivatives and norms of operators and 
 vectors.
   
We establish an explicit correspondence between (i) symmetric and 
positive measures  $\rho$  on one hand, and (ii) a certain set of 
generalized Markov transition measures $P$ on the other. More precisely,
 by a generalized Markov transition measure we mean a measurable and 
 measure-valued function $P$ on $(V, \B)$, such that for every $x$
in $V$, $P(x, \cdot)$ is a probability measure on $(V, \B)$. From the
 generalized Markov transition $P$, we introduce an associated operator, 
 also denoted by $P$. Its spectral theoretic properties refer to a certain
  $L^2$ space, and they will be made precise in Section 
\ref{subsect_Operators}.  
  
In addition to the operator $P$, we shall also consider a natural transfer 
operator $R$  (the choice of the letter ``$R$'' is for David Ruelle who 
initiated a variant of our analysis in the context of statistical mechanics);  
and a measure theoretic Laplacian, or Laplace operator. In the special case 
when $V$ is countably discrete, our Laplace operator will be analogous to 
a family of more standard discretized classical Laplace operators.
For related results on transfer operators, see e.g. 
\cite{AlpayJorgensenLewkowicz2018, Baggett-et-al2009, 
Baggett-et-al2012, BrattelliJorgensen1999, 
Cioletti_et_al2017,  DutkayJorgensen2014, JiangYe2018, 
 JorgensenTian_2017, Jorgensen2001, 
JorgensenPedersen1998, Ruelle1989, Ruelle1992}.

\textbf{New results.} 
It is important to note that our setting is not restricted to the case of 
finite measures. In fact, in our discussion of Markov transition dynamics, 
important examples simply will not allow finite covariant measures. We 
recall that the theory of weighted networks can serve as a discrete 
analog of our measurable settings, see 
\cite{BezuglyiJorgensen2018} where this analogy was discussed in
detail. The corresponding  symmetric measure on the edge set $E$ is 
$\sigma$-finite as well as the counting measure $\mu$ on the set 
of vertices $V$. Our definitions of the energy space $\h_E$, Markov 
operator $P$,  and 
the graph Laplace operator $\Delta$ are direct translations of the 
corresponding definitions for weighted networks.

 To the best of our
knowledge, such interpretations of these objects   have  not been 
considered earlier. We stress that our approach to Markov processes 
generated by $\sigma$-finite symmetric measures leads with 
necessity to the study of Markov transition operators defined on
infinite $\sigma$-finite measure spaces. The existing literature on 
Markov processes is devoted mostly to the case of probability measure 
spaces, see, e.g., \cite{LyonsPeres2016, Nummelin1984, Revuz1984}. 

The notion of Borel equivalence relation defined on a standard
Borel space illustrates our setting, and it can be viewed as a rich source
of various examples.  We refer to  the following books and articles:
\cite{ConleyMiller2016, ConleyMiller2017, CornfeldFominSinai1982, 
DoKuchmentOng2017, FeldmanMooreI_1977,  Kanovei2008, 
Kechris1995,  Lehn1977}. 

More applications of measurable setting for the study of Markov 
processes and Laplacians are given in \cite{BezuglyiJorgensen2018}.
We mention here the theory of graphons, Dirichlet forms, and the
theory of determinantal measures. 

With our starting point, a choice of a fixed symmetric and positive 
measure  $\rho$  on a product space, we will then have four natural 
Hilbert spaces, three are just $L^2$ spaces, $L^2(\rho)$, and two $L^2$ 
spaces referring to the marginal measure $\mu$. The fourth Hilbert space 
is different. We call it the finite energy Hilbert space $\h_E$. Its use is 
motivated by potential theory, and it has a more subtle structure 
among the considered Hilbert 
spaces. Given $\rho$, we introduce an associated energy Hilbert space, 
denoted $\h_E$, but depending on the initially given $\rho$. This energy 
Hilbert space $\h_E$ is not directly linked to the initial $L^2$ spaces, and 
its properties are quite different. Nonetheless, the energy Hilbert space 
$\h_E$ will 
play a key role in our analysis in the main body of our paper. There are 
many reasons for this. For example, non-constant harmonic functions will 
not be in $L^2$; but, in important applications, they may be in $\h_E$; 
we refer to the latter as finite energy harmonic functions. A second 
reason for $\h_E$ is that it plays a crucial role in our introduction of a 
generalized 
Green's function. The latter stands in relation to our Laplace operator in a 
way that is parallel to more classical settings of Green's functions from 
potential theory. A third reason for $\h_E$ is its use in our analysis of 
path-space dynamics for the Markov transition system, mentioned above.

\textbf{Organization.}
Our \textit{main results} are proved in Theorems  \ref{prop embedd J},
\ref{lem_ P eq P'}, \ref{thm energy rho and rho'}, \ref{prop_reversible},
\ref{prop from A to B}, \ref{prop on G_A}, and \ref{prop connectedness}.

The paper is organized as follows. Section \ref{sect Prelim} contains our
 basic definitions and preliminary results. We discuss here the concepts
 of standard Borel and standard measure spaces, kernels,
irreducible symmetric measures, and disintegration. The transfer operator 
$R$, Markov operator $P$, and graph Laplacian $\Delta$ are defined
in Section \ref{subsect_Operators}. We collected a number of results 
 about the spectral properties of these operators that were 
 proved in \cite{BezuglyiJorgensen2018}. Also the reader will find 
 the definition of the finite energy Hilbert space $\h_E$, several 
 results about the structure of the space $\h_E$ and the norm of 
 functions from $\h_E$. We 
 consider also the embedding operator $J$ and prove that $J$ is an
 isometry. In Section \ref{sect equivalence}, we consider the 
 equivalence of Markov operators and the Laplacians 
 generated by equivalent symmetric measures $\rho$ and $\rho'$. It
 turns out that, for equivalent symmetric measures $\rho$ and $\rho'$, 
 there exists an isometry for the corresponding energy Hilbert spaces 
 $\h_E(\rho)$ and $\h_E(\rho')$. The notion of reversible Markov
 processes is discussed in Section \ref{sect Transient}. We relate various
 properties of  the operator $P$ (such as self-ajointness) to this notion 
 and to the notion of a symmetric measure. A number of results 
 about Markov operators acting in the $L^2$ spaces and energy space
 $\h_E$  are  proved in this section. Section \ref{subsect Path space MP}
focuses on the case of a transient Markov processes defined
by a Markov operator $P$. We define the path-space measure 
$\mathbb P$ and Green's function $G(x, A)$, and  we discuss their 
properties. Section \ref{sect discretization} is devoted to construction
of a sequence of discrete weighted networks which can be used
to approximate  the objects considered for the measurable setting.

\section{Basic definitions and symmetric measures}\label{sect Prelim}

In this section, we briefly describe our main setting and introduce 
the most important notation. We also recall several results 
from  \cite{BezuglyiJorgensen2018} which will be used here. 

\subsection{Standard Borel and measure spaces}
Suppose $V$ is a \textit{Polish space}, i.e., $V$ is a separable completely
 metrizable topological space. Let $\B$ denote the $\sigma$-algebra of
 Borel sets   generated by open  sets of $V$. Then $(V, \B)$ is called a 
\textit{standard Borel space}. The theory of standard Borel spaces is 
discussed in many  recent books,  see e.g., 
 \cite{Gao2009, Kanovei2008, Kechris1995, Kechris2010} and papers
 \cite{Chersi1989, Loeb1975}. We recall that 
all uncountable  standard Borel spaces are Borel isomorphic, so that one
 can use any convenient  realization of the space $V$ working
 in the category of measurable spaces. If $\mu$ is a
  continuous (i.e., non-atomic) positive  Borel  measure  on  $(V, 
 \mathcal B)$, then  $(V, \mathcal B, \mu)$ is called a \emph{standard 
 measure space}. Given $\sms$, we will call $\mu$ a measure for brevity. 
 As a rule, we will deal  with non-atomic  $\sigma$-finite positive 
 measures on $(V, \B)$ (unless the opposite is clearly indicated) which 
 take values in the extended real line $\ol \R$. 
 We  use the name of \textit{standard measure space} for both finite and 
 $\sigma$-finite measure spaces. 
Also the same notation, $\B$, is applied for the $\sigma$-algebras of 
Borel sets and measurable  sets of a standard measure space. It 
should be clear from the context what $\sigma$-algebra is considered.  
Working with a measure space $\sms$, we always assume that  $\B$ is 
 \textit{complete}  with respect  to $\mu$. By $\mathcal F(V, \B)$. 
 we denote  the space of real-valued bounded Borel functions on $(V, \B)$.
For  $f \in  \mathcal F(V, \B)$ and a Borel measure $\mu$ on $(V, \B)$, 
 we write 
 $$
 \mu(f) = \int_V f\; d\mu.
 $$

All objects, considered in the context of measure spaces (such as sets,
functions, transformations, etc), are determined  by modulo sets 
of zero measure.  In most cases, we will  
implicitly use this mod 0 convention not mentioning the sets of 
zero measure explicitly. 

In what follows, we will use  (in most cases implicitly)  the notion of
 \textit{measurable fields}. Given 
a measure space $\sms$, it is said that $x \mapsto A_x \in \B$ is a
 \textit{measurable field of sets} if the set 
 $$
 \bigcup_{x\in V} \{x\} \times A_x \in \B \times \B.
 $$
Similarly, one can define a \textit{measurable field of measures} 
$x \to \mu_x$ on 
$\VB$ requiring  $x \mapsto \mu_x(A)$ to be a measurable function for
any $A \in \B$.
 
Consider a $\sigma$-finite continuous measure $\mu$ on a standard
Borel space  $(V, \B)$.  We denote by 
\be\label{eq Bfin}
\Bfin = \B_{\mathrm{fin}}(\mu)  = \{ A \in \B : \mu(A) < \infty\}
\ee
the algebra of Borel sets of finite measure $\mu$. Clearly, any 
$\sigma$-finite  measure 
$\mu$ is uniquely determined by its values on $\Bfin(\mu)$. 

The linear space of simple function over sets from $\Bfin(\mu)$ is denoted
by
\be\label{eq Dfin}
\ba
\mathcal D_{\mathrm{fin}}(\mu) := & 
\left\{ \sum_{i\in I} a_i \chi_{A_i}  :
 A_i \in \Bfin(\mu),\ a_i \in \mathbb R,\ | I | <\infty\right\}\\
  = &\  \mbox{Span}\{\chi_A : A \in \Bfin(\mu)\}, 
  \ea
\ee
will play an important role in our work since simple functions
from  $\Dfin(\mu)$ form a norm  dense subset in  $L^p(\mu)$-space,
$ p \geq 1$.

\subsection{Symmetric measures, kernels, and disintegration} 
\label{subsect m sp symm m}

\begin{definition}\label{def symmetric set}
Let $E$ be an uncountable Borel subset of the Cartesian product 
$(V \times V, \B\times \B)$  such that:

(i)  $(x, y) \in E\  \Longleftrightarrow \ (y, x) \in E$, i.e. $\theta(E) = E$
where $\theta(x, y) = (y,x)$ is the flip automorphism;

(ii) $ E_x := \{y \in V : (x, y) \in E\} \neq \emptyset, \ \   \forall x \in X$;

(iii) for every $x \in V$, 
 $(E_x, \B_x)$ is a standard Borel space where $\B_x$ is 
 the $\sigma$-algebra of Borel sets induced on $E_x$ from $(V, \B)$.
 
We call  $E$ a \textit{symmetric set}. 
\end{definition}

It follows from (ii) and (iii) that the projection of the symmetric set $E$ on
 each margin of the product space $\vv$ is $V$.

We observe that  conditions
(ii) and (iii) are, strictly speaking,  not related to the symmetry property;
 they are  included in Definition \ref{def symmetric set} for convenience,
so that we will not have to make additional assumptions. Condition (iii) 
assumes two cases: the Borel space $E_x$ can be countable or
 uncountable. We focus mostly on uncountable Borel standard spaces.

There are several natural examples of symmetric sets related to dynamical 
systems. We mention here the case of a  \textit{Borel equivalence relation}
$E$ on a standard Borel space $(V, \B)$. By definition,  $E$ is a Borel 
subset  of $V\times V$ such that $(x, x) \in E$ for all $x \in V$, 
$(x, y)$ is in $E$ iff $(y, x) $ is in $E$, and $(x, y) \in E, (y, z) \in E$ 
implies that $(x, z)\in E$. Let $E_x = \{ y \in V : (x, y) \in E\}$, then 
$E$ is partitioned into ``vertical fibers'' $E_x$.  In particular, it can be the
case when every $E_x$ is countable. Then $E$ is called a \textit{countable
Borel equivalence relation. }

\medskip 
We say that a symmetric set $E$ is \textit{decomposable} if there exists
an uncountable Borel  subset $A \subset V$ such that
\be\label{eq decomposable}
E \subset (A \times A) \cup (A^c \times A^c),
\ee
where $A^c = V \setminus A$. 
\medskip

The meaning of this definition can be clarified for Borel equivalence
 relations: if $E$ satisfies \eqref{eq decomposable}, then the set $A$ is 
 \textit{$E$-invariant}.

We recall several definitions and facts about kernels defined on a
 measurable  space, see e.g. \cite{Nummelin1984}, \cite{Revuz1984}.
  Given a standard measure space $(V, \B)$, we define a 
  \textit{$\sigma$-finite kernel} $k$ as a function $k:  
  V \times \B \to \ol \R_+$ (where $\ol\R_+$ is the extended real line)
    such that

(i) $x \mapsto k(x, A)$ is measurable for every $A \in \B$;

(ii) for any $x\in V$, $k(x, \cdot)$ is a $\sigma$-finite measure on 
$(V, \B)$.
\medskip

A kernel $k(x, A)$ is called \textit{finite} if $k(x, \cdot)$ is a finite
 measure on  $(V, \B)$ for every $x$. We will also use the notation 
 $k(x, dy)$ for the measure on $\VB$.

The definition of a finite kernel can be used to define new measures on 
the measurable spaces $(V, \B)$ and $\vv$. 

Given a $\sigma$-finite measure space $\sms$ and a finite kernel 
$k(x, A)$, we set 
$$
\kappa (A) = \int_V k(x, A)\; d\mu(x).
$$
Then $\kappa$ is a $\sigma$-finite measure on $(V, \B)$ (which is also 
called a \textit{random measure} in the literature). Obviously, $\kappa$ is
  absolutely continuous with respect to $\mu$ but not, in general,
   equivalent to $\mu$.
  
For a kernel $k$ as above, one can define inductively the sequence 
of kernels $(k^n : n \geq 1)$ by setting
\be\label{eq_powers of k}
k^n(x, A) = \int_V k^{n-1}(y, A)\; k(x, dy),\qquad n > 1.
\ee

Following \cite{Nummelin1984}, we formulate definitions of main 
properties of a kernel $k$. We say that a \textit{set $A \in \B$ 
is \textit{attainable}  from $x\in V$} if there exists  $n\geq 1$ such that 
$k^n(x, A)  >0$, in symbols, we write $x \rightarrow A$. A set $A \in \B$
 is called  \textit{closed} for the kernel $k$ if $k(x, A^c) = 0$ for all $x\in
  A$. If $A$ is closed, then it follows from \eqref{eq_powers of k} that
 $k^n(x, A^c) = 0$ for any $n \in \N$ and   $x \in A$.  Hence, $A$ is
  closed if and only if  $ x \nrightarrow   A^c$. 
  
A kernel $k = k(x, A)$   is  called \textit{Borel indecomposable} on $(V,
 \B)$ if there do not exist two disjoint non-empty closed subsets $A_1$ 
 and $A_2$.  
     
Let $F_x\in \B$ be the support of the measure $k(x, \cdot)$, that is $k(x, 
V\setminus F_x) = 0$. By $\wt F_x$, we denote the set $\{x\} \times 
K_x \subset V \times V$. Then the formula 
$$
k (A\times B) = \int_A \wt k(x, B)\; d\mu(x)
$$
defines a $\sigma$-finite measure on $\vv$ where $\wt k(x, \cdot)
 = (\delta_x \times k)(x, \cdot)$. The support of $k$ is the set
$$
F := \bigcup_{x\in V} \wt F_x.
$$

We will use below slightly simplified notation identifying the 
sets $F_x$ and $\wt F_x$ and the measures $k(x, A)$ and $\wt k(x, A)$.
It will be clear from the context what objects are considered. 
\medskip

As mentioned in Introduction, our approach is based on the study of
 \textit{symmetric measures} defined on $\vv$, see Definition
 \ref{def symm measure rho}.  
 We show that   every  measure $\rho$ on $\vv$ generates a kernel 
 $x \to \rho_x(A), A\in \B$.  This observation  is 
based on the concept of  \textit{disintegration} of the measure $\rho$. 
We recall here this construction.

Denote by $\pi_1$ and $\pi_2$ the projections from $V \times V$ onto     
the first and second factor, respectively.  Then  $\{\pi_1^{-1}(x) :
 x \in V\}$
and $\{\pi_2^{-1}(y) : y \in V\}$ are  the\textit{ measurable partitions} of 
$V \times V$ into vertical and horizontal fibers, see \cite{Rohlin1949, 
CornfeldFominSinai1982, BezuglyiJorgensen2018} for more information on
 properties of measurable partitions. The case of probability
 measures was studied by Rokhlin in \cite{Rohlin1949}, whereas the 
 disintegration of $\sigma$-finite measures has been considered somewhat
 recently. We
 refer to a result from  \cite{Simmons2012} whose formulation
 is adapted to our needs. 

\begin{theorem}[\cite{Simmons2012}] \label{thm Simmons} 
For a $\sigma$-finite measure space $(V, \B, \mu)$, 
let $\rho$ be a $\sigma$-finite measure on $(V\times V, \B\times \B)$ 
such that $\rho\circ \pi_1^{-1} \ll \mu$. Then there exists a unique 
system of conditional $\sigma$-finite measures $(\wt\rho_x)$ such that
$$
\rho(f)  = \int_V \wt\rho_x(f)\; d\mu(x), \ \ \ f\in \mathcal F(V	\times V, 
\B\times \B).
$$
\end{theorem}

In the following remark we collect several facts that clarify the essence of 
the defined objects.

\begin{remark}\label{rem symm meas}
(1) The condition of  Theorem \ref{thm Simmons} assumes that a measure
$\mu$ is prescribed on the Borel space $\VB$. If one begins with a 
measure $\rho$ on $\vv$, then the measure $\mu$ arises as the
projection of $\rho$ on $\VB$,  $\rho\circ \pi_1^{-1} = \mu$.

(2) Let $E$ be a Borel symmetric subset of $\vv$, and let $\rho$ be 
a measure on $\vv$ satisfying the condition of Theorem \ref{thm Simmons}.
Then $E$ can be  partitioned into the fibers $\{x\} \times E_x$. By
 Theorem \ref{thm Simmons}, there exists 
a unique  system of  conditional measures $\wt\rho_x$ such that, for any 
 $\rho$-integrable function $f(x, y)$, we have
\be\label{eq disint for rho}
\iint_{V\times V} f(x, y) \; d\rho(x, y)  = \int_V \wt\rho_x(f) \; d\mu(x). 
\ee
It is obvious that, for $\mu$-a.e. $x\in V$, $\mbox{supp}(\wt
\rho_x) = \{x\} \times E_x$ (up to a set of zero measure). To simplify the
 notation, we will  write 
$$
\int_V f\; d\rho_x \ \ \mathrm{and}  \ \iint_{V \times V} f\; d\rho
$$ 
though the measures $\rho_x$ and $\rho$ have the supports $E_x$ and 
$E$,  respectively. 

 (3) It follows from  Theorem \ref{thm Simmons} that the measure $\rho$
 determines the  measurable field of
sets $x \mapsto E_x \subset V$ and  measurable field of $\sigma$-finite
Borel  measures $x \mapsto \rho_x $ on $(V, \B)$,
 where  the measures $\rho_x$ are defined by the relation
\be\label{eq rho_x def}
\wt\rho_x = \delta_x \times \rho_x. 
\ee
Hence, relation (\ref{eq disint for rho}) can be also written in the following 
form,  used in our subsequent computations,
\be\label{eq disint formula}
\iint_{V\times V} f(x, y) \; d\rho(x, y)  = \int_V \left(\int_V f(x, y) \; 
d\rho_x(y)\right)\; d\mu(x).
\ee
In other words, we have a measurable family of measures $(x \mapsto 
\rho_x)$,  
and it defines a new measure $\nu$ on $(V, \B)$ by setting
\be\label{eq def of nu}
\nu(A) := \int_V \rho_x(A)\; d\mu(x), \quad A \in \B.
\ee
Remark that the measure $\rho_x$ 
is defined on the subset $E_x$ of $(V, \B)$, $x \in V$.
\end{remark}

\begin{definition}\label{def symm measure rho}
 Let $(V, \B)$ be a standard Borel space.
We say that a measure $\rho$ on $(V \times V, \B\times B)$ is
 \textit{symmetric} if 
 $$
 \rho(A \times B) = \rho(B\times A), \ \ \ \forall A, B \in \B.
 $$ 
 In other words, $\rho$ is invariant with respect to the flip automorphism
$\theta$.
 \end{definition}

The following remark contains  natural properties of symmetric
measures. Some of them were proved in \cite{BezuglyiJorgensen2018},
the others are rather obvious. 

\begin{remark}\label{rem on symm meas}
(1) If $\rho$ is a symmetric measure on $(V \times V, \B\times \B)$,
 then the support of $\rho$, the set $E = E(\rho)$, is symmetric mod 0. 
 Here $E(\rho)$ is defined up to a set of zero measure by the relation 
$\rho((V \times V) \setminus E) = 0$.

(2) We consider the symmetric measures whose supporting 
sets $E$ satisfy Definition \ref{def symmetric set}. In other words, we 
require that, for every $x \in V$, the set $E_x\subset E$ is uncountable 
and therefore is a standard Borel space. The case 
when $E_x$ is countable arises, in particular, when $E$ is a Borel countable
equivalence relation on $(V, \B)$. The latter was considered in
\cite{BezuglyiJorgensen2018}. For countable sets $E_x, x \in V$, we can
take $\rho_x$ as a finite measure which is equivalent to the counting
measure, see, e.g.  \cite{FeldmanMooreI_1977, FeldmanMooreII_1977, 
KechrisMiller2004} for details. 

(3) In general, the notion of a symmetric measure is defined in the context 
of standard Borel spaces $(V, \B)$ and $(V\times V, \B\times \B)$. 
But if a $\sigma$-finite measure $\mu$ is given on $(V, \B)$, then we need 
to include an additional relation between the projections of $\rho$ on $V$
 and the measure  $\mu$. 
Let $\pi_1 : V \times V \to V$ be the projection on the first coordinate. 
We require that the symmetric measure must satisfy
the property $\rho\circ \pi_1^{-1} \ll \mu$, see Theorem 
\ref{thm Simmons}.

(4) The symmetry of the set $E$ allows us to define 
a ``mirror'' image of the measure $\rho$. Let $E^y := \{x \in V : (x, y) \in 
E\}$, and let $(\wt\rho^y)$ be the system of conditional measures with 
respect to the partition of $E$ into the sets $E^y \times \{y\}$. Then, 
for the measure 
$$
\wt \rho = \int_V \wt \rho^y d\mu(y),
$$
the relation $\rho = \wt \rho$ holds. 

(5) It is worth noting that, in general, when a measure $\mu$ is defined on 
$(V, \B)$, the set $E(\rho)$ do not need to be 
a set of positive measure with respect to the product measure
$\mu\times \mu$. In other words, we admit both cases: (a) $\rho$ is 
equivalent to $\mu\times\mu$,  (b) $\rho$ and
$\mu\times\mu$ are mutually singular. 
\end{remark}

\textit{\textbf{Assumption 1}}.  In this paper, we consider the class of 
symmetric measures $\rho$ on $(V\times V, \B\times \B)$ which satisfy 
the following property: 
\be\label{eq c finite}
0 < c(x) := \rho_x(V) <\infty, \ \ \  \ \ \ \mu\mbox{-a.e.}\ x\in V,
\ee
where $x \mapsto \rho_x$ is the measurable field of measures arising
in Theorem \ref{thm Simmons}.

Moreover, in most statements, we will assume that $c(x) \in
 L^1_{\mathrm{loc}} (\mu)$, i.e.,
$$
\int_A c(x)\; d\mu(x) < \infty, \qquad \forall A \in \Bfin(\mu).
$$
This property of the function $c(x)$ is natural because it corresponds to
local finiteness of graphs in the theory of weighted (electric) networks.  
In several statements, we will require that
$$
\left(\forall A\in \Bfin(\mu),\ \  \int_A c^2\; d\mu < \infty\right) \
 \Longleftrightarrow \  c \in \lc2.
$$ 
We observe also that the case when the function $c$  is bounded leads
to bounded Laplace operators and is not interesting for us.

Relation (\ref{eq def of nu}) defines the measure $\nu$ such that  the
 measures $\mu$ and $\nu$  are equivalent. It is stated  in Lemma 
 \ref{lem symm measure via int} that $c(x)$ is the Radon-Nikodym 
 derivative of $\nu$ with respect to $\mu$. If we
want to reverse the definition and use $\nu$ as a primary measure, then 
we need to require that the function $c(x)^{-1}$ is locally integrable 
with respect to $\nu$. 
\medskip

The following (important for us) fact follows from the definition of
 symmetric
measures. We emphasize that formula (\ref{eq formula fo symm meas})
will be used repeatedly in many proofs. 

\begin{lemma} \label{lem symm measure via int}
(1) For a symmetric measure $\rho$ and any bounded Borel function
$f$ on $(V\times V, \B\times \B)$, 
\be\label{eq formula fo symm meas}
\iint_{V \times V} f(x, y) \; d\rho(x, y) = \iint_{V \times V} f(y, x) \; 
d\rho(x, y). 
\ee
Equality (\ref{eq formula fo symm meas}) is understood in the sense of 
the extended real line, i.e., the infinite value of the integral is allowed. 

(2) Let $\nu$ be defined as in (\ref{eq def of nu}). Then 
$$
d\nu(x) = c(x) d\mu(x).
$$ 
\end{lemma}

\subsection{Irreducible symmetric measures}
We now relate the notions of symmetric measures and kernels. 
It turns out that one can associate a finite kernel $\mc K(\rho) = \mc K$
to any symmetric measure $\rho$ on $\vv$. For this, we use the
 disintegration of $\rho$ according to Theorem \ref{thm Simmons},
 $\rho = \int_V \rho_x \; d\mu(x)$, and set 
  $x\to \mc K(x, A) = \rho_x(A)$.

The definition of sets attainable from $x \in V$ and that of
decomposable sets, given above in the context of Borel spaces, can be
translated to the case of measure spaces.  Below we define the notion of  
an \textit{irreducible symmetric measure} which will be extensively used in 
the paper. 

 \begin{definition}\label{def_irreducible}
(1)  A kernel $x \to k(x, \cdot)$ is called \textit{irreducible with respect 
to a $\sigma$-finite  measure  $\mu$ on $\VB$ ($\mu$-irreducible)} if, 
for any set $A$ of 
positive measure $\mu$ and $\mu$-a.e. $x\in V$, there exists some $n$ 
such that   $k^n(x, A) >0$,  i.e., any set $A$ of positive measure  is
 attainable from $\mu$-a.e. $x$, $x \rightarrow A$. 

(2) A \textit{symmetric measure} $\rho$ on $\vv $ is called
  \textit{irreducible} if the corresponding kernel $\mc K(\rho) : x \to
   \rho_x(\cdot)$ is $\mu$-irreducible where $\mu$ is the projection of 
 measure  $\rho$. 

(3) A symmetric measure $\rho$ (or the kernel $ x \to
   \rho_x(\cdot)$) is called $\mu$-\textit{decomposable} if 
 there exists a Borel subset $A$ of $V$ of positive measure $\mu$ 
 such that 
\be\label{eq decomposable 1}
 E \subset (A\times A) \cup (A^c \times A^c)
 \ee
 where $A^c = V \setminus A$ is also of positive measure. 
 Otherwise, $\rho$ is called  \textit{indecomposable}. 
\end{definition} 

Every kernel $k$, defined on $\VB$, generates the 
\textit{potential kernel}
$$
G(k)(x, A) := \sum_{n=0}^\infty k^n(x, A)
$$
where $k^0(x, A) = \chi_A(x)$. In general, the kernel $G$ may be
 degenerated admitting only the values 0 and $\infty$. We will discuss
 below the role of $G$ in the case of transient Markov processes. 

\begin{lemma} Let $ \rho$ be a symmetric measure on $\vv$ with 
the kernel $\mc K(x, A) = \rho_x(A)$. Suppose that the support of 
$\rho$, the set $E$, satisfies relation \eqref{eq decomposable 1} 
where $\mu(A) >0$ and $\mu(A^c) >0$, i.e. the kernel $x \mapsto \rho_x(A)$ is $\mu$-decomposable. Then the sets $A$ and $A^c$ 
are closed and  $x\mapsto \rho_x(A)$ is a $\mu$-reducible kernel. 
The converse statement also holds. 
\end{lemma}

\begin{proof}
The first result follows directly  from the definitions given above in this
 subsection. To see that the converse is true, it suffices to note that, for
 any set $B$ of positive measure, the compliment $\wh B^c$ of the set 
 $$
 \wh B := B \cup \{ x \in V : x \to B\}
 $$  
 is either of zero measure, or closed (recall that $x \to B$ means that 
 there exists $n$ such that $\mc K^n(x, B) >0$) . If $\rho$ is reducible,
  then there 
 exists a set $A, \mu(A) >0,$  such that the closed set $\mu(\wh A^c)$ 
 has positive measure.  The existence of such a set implies that the 
 measure $\rho$ is decomposable.
\end{proof}

It is obvious from this lemma  that a decomposable symmetric
 measure $\rho$ cannot be irreducible. 
It was proved in \cite{BezuglyiJorgensen2018} that the definitions of an
irreducible measure and irreducible kernel agree, see Theorem 
\ref{prop from A to B} below.
 
 By definition, the projection of the support of an irreducible measure 
 $\rho$ is the set $V$. Irreducibiliity of symmetric measures means
  irreducibility of a corresponding Markov process, see details in 
\cite{BezuglyiJorgensen2018}. 
\medskip

In the following statement, we give another approach to the notion
of irreducible symmetric measures.  Let  $\rho$ be a symmetric measure
on $\vv$. We use the support of the fiber measure $\rho_x, x\in V$, 
to characterize an irreducible measure in different terms.

For any fixed $x\in V$, we define a sequence of subsets: 
$A_0(x) = \{x\}$,  $A_1(x) = E_x,$ 
$$
  A_n(x) =  \bigcup_{y\in A_{n-1}(x)} E_y, \ \ \ n \geq 2.
$$ 
Recall that $E_x$ is the support of the measure $\rho_x$, and $E_x$ can 
be identified with the vertical section of the symmetric set $E$. Note that 
all the sets $A_n(x)$ are in $\B$ as $x\to E_x$ is a measurable  field
 of sets.

\begin{lemma} \label{lem-irr measure}
Given $\sms$, a symmetric  measure $\rho$ is irreducible if
 and only if for $\mu$-a.e. $x \in V$  and any set $B\in \B$  of positive
  measure there exists $n\geq 1$  such that 
  \be\label{eq_irr measure via A_n}
  \mu(A_n(x) \cap B) >0. 
  \ee
\end{lemma}

\begin{proof}
Indeed, the property formulated in \eqref{eq_irr measure via A_n}
 is another 
form of $k^n(x, B) >0$ where the kernel $k$ is defined by $x\to \rho_x$.
\end{proof}

Various aspects of symmetric measures are also discussed in 
\cite{ChenRenYang2017, AlimoradFakharzadeh2017}. In particular, one
can observe that if symmetric measures $\rho$ and $\ol \rho$ are 
equivalent, then they are simultaneously either irreducible or not.

\section{Linear operators and Hilbert spaces associated to  symmetric
measures} 
\label{subsect_Operators}

\subsection{Symmetric operator $R$, Markov operator $P$, and 
Laplacian $\Delta$}
Suppose $k : V \times \B \to \R_+$ is a finite kernel defined on a 
standard Borel space $\VB$. Then it defines
a linear positive (see Remark \ref{rem_on operators}) operator $P(k)$ 
 which is determined by  the kernel $k$:
\be\label{eq_def P via k}
P(k)(f)(x) := \int_V f(y)\; k(x, dy).
\ee
It can be easily seen that, for the kernels $k^n$ (see 
\eqref{eq_powers of k}), the operator $P(k^n)$, defined as in
\eqref{eq_def P via k}, satisfies the property: 
$$P(k^n) = P(k)^n, n \in \N.
$$

We consider in this section the kernel $\mc K(\rho)$ generated by a
 symmetric measure $\rho$, i.e., $\mc K(x, A) = \rho_x(A)$. 
 
Let $\sms$ be a $\sigma$-finite measure space, and $\rho$ a symmetric 
measure on $\vv$ supported by a symmetric set $E$. Let  $x\mapsto 
\rho_x$ be the measurable family of measures on $\VB$ that 
 \textit{disintegrates} 
$\rho$. Recall that, by Assumption 1, the function  $c(x) = \rho_x(V)$ is
 finite for $\mu$-a.e. $x$. As discussed above in Subsection 
 \ref{subsect m sp symm m}, the measure $\rho$ produces a 
finite kernel $\mc K(\rho)$ which we use to define the following
 operators. 

\begin{definition}\label{def R, P, Delta}
For  a symmetric measure $\rho$  on $\vv$, we
define three linear operators $R, P$ and $\Delta$ acting on the space of
bounded Borel functions $\FVB$.\\

(i) The \textit{symmetric operator}$R$:
\be\label{eq def of R} 
R(f)(x) := \int_V f(y) \; d\rho_x(y) = \rho_x(f). 
\ee 

(ii) The\textit{ Markov operator} $P$:
$$
P(f)(x) = \frac{1}{c(x)}R(f)(x)
$$
or
\be\label{eq formula for P}
 P(f)(x) := \frac{1}{c(x)}  \int_V f(y) \; d\rho_x(y) = \int_V f(y) \; 
 P(x, dy)
 \ee
 where  $P(x, dy)$ is the probability measure obtained by normalization
 of $d\rho_x(y)$, i.e. 
 $$
 P(x, dy) := \frac{1}{c(x)}d\rho_x(y).
 $$
 In other words, the Markov operator $P$ defines the measurable 
 field $x \mapsto P(x, \cdot)$ of \textit{transition probabilities} on the
  space  $\VB$, or a \textit{Markov process}.
 
 (iii) The \textit{graph Laplace operator} $\Delta$:
 \be\label{eq def of Delta}
\Delta(f)(x) := \int_V (f(x) - f(y)) \; d\rho_x(y)
\ee
or 
\be\label{eq Delta via R}
\Delta(f) = c(I - P)(f) = (cI  - R)(f).
\ee
Using (\ref{eq c finite}), we can write the operator $\Delta$ in more
symmetric form:
$$
\Delta(f) = R(\mathbbm 1)f  - R(f)
$$
where $\mathbbm 1$ is a function identically equal to $1$,

\end{definition}

\begin{remark}[$R$ as a transfer operator]
It is worth noting that the operator $R$ can be treated as a transfer
 operator (see e.g. \cite{BJ_book} and the  literature cited there). 
 
Let $\sms$ be a standard measure space, and let 
$\sigma$ be a surjective endomorphism  of $X$. Consider  the 
partition  $\xi$ of $X$ into the orbits of $\sigma$: $y \in Orb_\sigma(x)$
if there are non-negative integers $n,m$ such $\sigma^n(y) = 
\sigma^m(x)$. Let  the partition $\eta$ be the  measurable hull of $\xi$. 
 Take the system of conditional  measures $\{\mu_C\}_{C \in \xi}$ 
corresponding to the partition  $\eta$ (see Theorem \ref{thm Simmons}). 

We define a  transfer operator $R$ on the standard 
measure space  $(V, \B, \mu)$ by setting 
\be\label{eq TO via cond syst meas Intro}
R(f)(x) := \int_{C_x} f(y)\; d\mu_{C_x}(y)
\ee
where $C_x$ is the element of $\eta$ containing $x$. The domain of 
$R$ is $L^1(\mu)$ in this example.

As was shown  in \cite{BJ_book},
the operator $R : L^1(\mu) \to L^1(\mu)$ defined by 
(\ref{eq TO via cond syst meas Intro}) is a \textit{transfer operator},
 i.e., it satisfies the relation
$$
R((f\circ\sigma) g)(x) =  f(x) (Rg)(x).
$$

To see that our definition of the operator $R$ given in 
\eqref{eq def of R} agrees with  \eqref{eq TO via cond syst meas Intro},
it suffices to take the measurable partition $\eta$ of $\VtV$ into subsets
$\{\pi_1^{-1}(x) : x \in V\}$ where $\pi_1$ is the projection of $\VtV$
onto $V$. 

\end{remark}

\begin{remark}\label{rem_on operators} 
In this remark we make several  comments about the basic
 properties of the operators $R$, $P$, and $\Delta$.
 
(1) The definition of each of the operators $R$, $P$, and $\Delta$
 depends on a symmetric measure $\rho$, and, strictly speaking,
they must be denoted as  $R(\rho)$, $P(\rho)$, and $\Delta(\rho)$.
Since most of our results are proved for a fixed measure $\rho$, we will
drop this variable. Below in this section, we discuss the relationships 
between $P(\rho)$ and $P(\rho')$ when $\rho$ and $\rho'$ are
 equivalent symmetric measures.

(2) The operators $R$ and $P$ are \textit{positive} in the sense that 
$R(f) \geq 0$ and $P(f) \geq 0$ 
whenever $f \geq 0$. Moreover, if $f = \mathbbm 1$, then
 $P(\mathbbm 1) = \mathbbm 1$ because 
every measure $P(x, \cdot) $ is probability. Hence, $P$  is a
 \textit{Markov operator. }

(3) The properties of the graph Laplace operator $\Delta$ are formulated
in Proposition \ref{prop prop of R, P, Delta}, which is given below. All 
statements from this theorem are proved in \cite{BezuglyiJorgensen2018}.
Other aspects of graph Laplace operators in the context of measure
spaces are discussed in \cite{SmaleZhou2007, SmaleZhou2009}.
 
(4) Since every measure  $\rho$  on $V \times V$  is uniquely determined
 by its values on a dense subset of functions, 
it suffices to define $\rho$ on the set of the so-called ``cylinder  
functions'' $(f \otimes g)(x, y) := f(x)g(y)$. This observation will be used
below when we prove a relation for cylinder functions first. 

(5) In general, a positive operator $R$ in $\FVB$ is called
 \textit{symmetric} if it satisfies the relation:
\be\label{eq symm in terms R}
\int_V f R(g)\; d\mu = \int_V R(f) g\; d\mu,
\ee
for any $f, g \in F(V, \B)$. It turns out that any symmetric operator $R$
defines a symmetric measure $\rho$. Indeed, the functional 
\be\label{eq symm R symm measure}
\rho: (f, g)\ \mapsto \ \int_V f(x) R(g)(x) \; d\mu(x), \qquad f, g \in 
F(V, \B),
\ee
determines a measure on $(V, \B)$ such that 
$$
\rho(A\times B) = \int_V \chi_A(x) R(\chi_B)(x)\; d\mu(x).
$$
As shown in \cite{BezuglyiJorgensen2018}, the operator $R$ is symmetric if 
and only if the measure $\rho$,  defined in 
(\ref{eq symm R symm measure}), is symmetric.
\end{remark}

In Definition \ref{def R, P, Delta}, we do not discuss domains of
the operators $R, P$, and $\Delta$. It depends on the space where an 
operator is considered. In the current paper, we work with $L^2$-Hilbert
 spaces 
defined by the measures $\mu, \nu$, and $\rho$. But the most 
intriguing is the case of the finite energy space Hilbert space $h_E$. 
We discuss the properties of this space as well as those of operators 
$\Delta$ and $P$ acting in $\h_E$ in the forthcoming paper 
 \cite{BezuglyiJorgensen}. On the other hand, we have already 
 proved a number 
 of results about these objects in \cite{BezuglyiJorgensen2018}. 
 We find it useful to give here the definitions and some formulas which 
 are  used below. 

We remark that the finite energy space $\h_E$, see Definition 
\ref{def f.e. space} can be viewed as a generalization of the energy 
space considered for discrete 
 weighted  networks. They have been  extensively studied during last
 decades. 
 
\begin{definition}\label{def f.e. space} Let $(V, \B, \mu) $ be a standard
 measure space with $\sigma$-finite measure $\mu$. Suppose that
 $\rho$ is a symmetric measure on the Cartesian product $(V\times V,
 \B\times \B)$.  We say that a Borel function
$f : V \to \mathbb R$ belongs to the \textit{finite energy space} 
$\mathcal H_E=  \mathcal H_E(\rho) $ if
\be\label{eq def f from H}
\iint_{V\times V}  (f(x) - f(y))^2 \; d\rho(x, y) < \infty.
\ee
\end{definition}

\begin{remark}\label{rem H depends on rho} 
(1) It follows from Definition \ref{def f.e. space} that   $\mathcal H_E$ 
is a vector space containing 
all constant functions. We identify functions $f_1$ and $f_2$ such that
$f_1 - f_2 = const$ and, with some abuse of notation, the quotient space
is also denoted by $\h_E$. So that, we will call elements $f$ of $\h_E$  
functions assuming that a representative of the equivalence class
$f$ is considered. 

(2) Definition \ref{def f.e. space} assumes that a symmetric irreducible 
measure $\rho$ is fixed on $(V\times V, \B\times \B)$. This means that
the space of functions $f$ on $(V, \B)$ satisfying (\ref{eq def f from H})
depends on $\rho$, and, to stress this fact, we will use also the 
notation $\h_E(\rho)$. 
\end{remark}
 
Define the norm in $\h_E$ by setting
\be\label{eq norm in H_E}
|| f ||^2_{\mathcal H_E} := \frac{1}{2} 
\iint_{V\times V}  (f(x) - f(y))^2 \; d\rho(x, y), \quad f \in \mathcal H,  
\ee
As proved in \cite{BezuglyiJorgensen2018}, $\h_E$ is a \textit{Hilbert 
space} with respect to the norm $|| \cdot ||_{\h_E}$. 
\medskip

The description of the  structure of the Hilbert space $\h_E$ is a very 
intriguing problem. We give here a few results proved in 
\cite{BezuglyiJorgensen2018}. 

\begin{theorem}\label{thm_stucture of energy space} 
Let $\rho$ be a symmetric measure on $\vv$ such that $\mu = \rho\circ
\pi_1^{-1}$. Suppose $c(x)= \rho_x(V)$ is locally integrable with respect 
to $\mu$. 

(1) For the measure $d\nu(x) = c(x)d\mu(x)$, we have 
$$
\mathcal D_{\mathrm{fin}}(\mu)  \subset 
\mathcal D_{\mathrm{fin}}(\nu) \subset \h_E.
$$
Moreover, if $A \in \Bfin(\nu)$, then 
\be\label{eq ||chi A||}
|| \chi_A ||^2_{\h_E} = \rho(A \times A^c) \leq \int_A c(x) \; d\mu(x)
= \nu(A),
\ee
where $A^c := V\setminus A$.

(2) For every $A \in \Bfin(\mu)$, one has $\| \chi_A \|_{\h_E}
 = \| \chi_{A^c} \|_{\h_E}$. The function $\chi_A $ is in $\h_E$
 if and only if either $\mu(A) < \infty$ or $\mu(A^c) < \infty$. 

(3) The finite energy space $\h_E$ admits the decomposition 
into the orthogonal sum 
\be\label{eq Royden}
\h = \ol{\mathcal D_{\mathrm{fin}}(\mu)} \oplus \h arm_E
\ee
where the closure of $\Dfin(\mu)$ is taken in the norm of the Hilbert
 space $\h_E$.
\end{theorem}

In the following statement we return to the $L^2$-spaces, and following
\cite{BezuglyiJorgensen2018}, we 
formulate a number of  properties of the operators, $R, P$, and 
$\Delta$ that clarify their essence. Here, we focus on the properties of
 these operators related
to $L^2$-spaces. In the next paper \cite{BezuglyiJorgensen},  
we will mostly consider  these operators acting in the finite energy space 
$\h_E$. 

\begin{proposition}\label{prop prop of R, P, Delta}
 Let $d\nu(x) = c(x) d\mu(x)$ be 
the $\sigma$-finite  measure on $(V, \B)$ where $\mu$ and $c(x) = 
\rho_x(V)$ are  defined  as above.
Let the operators $R, P$, and $\Delta$ be defined as in Definition
\ref{def R, P, Delta}. 

(1) Suppose that the function $x \mapsto \rho_x(A) \in L^2(\mu)$ for 
any $A \in \Bfin$.  Then $R$ is a symmetric unbounded operator  in $L^2(\mu)$, i.e.,   
$$
\langle g, R(f) \rangle_{L^2(\mu)} = \langle R(g), f \rangle_{L^2(\mu)}.
$$
If $c \in L^\infty(\mu)$, then $R :  L^2(\mu) \to 
L^2(\mu)$ is a bounded operator, and 
$$
||R||_{L^2(\mu) \to L^2(\mu)} \leq ||c||_{\infty}.
$$

(2) The operator $R : L^1(\nu) \to L^1(\mu)$ is contractive, i.e.,  
$$
||R(f)||_{L^1(\mu)} \leq ||f||_{L^1(\nu)}, \qquad f \in L^1(\nu).
$$
Moreover,   for any function $f \in L^{1} (\nu)$, 
the formula 
\be\label{eq c(x) and rho_x}
\int_V R(f) \; d\mu(x) = \int_V f(x) c(x) \; d\mu(x)
\ee
holds. In other words, $\nu = \mu R$, and 
$$
\frac{d(\mu R)}{d\mu}(x) = c(x).
$$

(3) The bounded operator $P: L^2(\nu) \to L^2(\nu)$ is self-adjoint.
Moreover, $\nu P = \nu$ where $d\nu(x) = c(x) d \mu(x)$. 

(4) The operator $P$ considered in the spaces $L^2(\nu)$ and $L^1(\nu)$ 
is  contractive, i.e., 
$$
|| P(f) ||_{L^2(\nu)} \leq || f ||_{L^2(\nu)}, \qquad || P(f) ||_{L^1(\nu)} \leq 
|| f ||_{L^1(\nu)}.
$$ 

(5) Spectrum of $P$ in $L^2(\nu)$ is a subset of $[-1, 1]$.

(6) The graph Laplace operator $\Delta : L^2(\mu)  \to L^2(\mu) $ 
is a positive definite essentially
self-adjoint operator with domain containing $\Dfin(\mu)$. Moreover, 
$$
|| f||^2_{\h_E} = \int_V f \Delta(f)\; d\mu
$$
when the integral in the right hand side exists. 
\end{proposition}

\begin{definition}\label{def_harmonic}
A  function $f\in \FVB$ is called \textit{harmonic}, if $Pf = f$.
Equivalently, $f$ is harmonic if $\Delta f = 0$ or $R(f) = cf$. Similarly, 
$h$ is \textit{harmonic} for a kernel $x\to k(x, \cdot)$ if 
$$
\int_V h(y) \; k(x, dy) = h(x).
$$
\end{definition}

\textit{Question}: 
As was mentioned above, the definition of operators $R(\rho), P(\rho)$,
 and $\Delta(\rho)$ is based
on a symmetric measure $\rho$ defined on $\vv$. Suppose that another 
symmetric measure, $\rho'$,  which is equivalent to 
$\rho$, is defined on $\vv$.  It would be interesting to find out 
what  relations between $(R(\rho), P(\rho), 
\Delta(\rho))$ and $(R(\rho'), P(\rho'), \Delta(\rho'))$ exist.
 Possibly, 
this question can be made more precise if we require that both $\rho$ and
$\rho'$ are supported by the same symmetric set $E$ and disintegrated
with respect to the same measure $\mu$ on $\VB$.
\\

\begin{remark} In our further results, the following sets of functions will play an important
 role. Let $\sms$ be a $\sigma$-measure space, and $\rho$ a symmetric
  measure  on $\vv$ satisfying Assumption 1. Then the measure 
  $d\nu(x) = c(x)  d\mu(x)$ is  on $\sms$ is equivalent to $\mu$ where
  $c(x) = R(\mathbbm 1)(x)$. We define $\Dfin(\mu)$ as in 
(\ref{eq Dfin}), and, similarly, we set 
$$
\Bfin(\nu) := \{ A\in \B : \nu(A) < \infty\},
$$
$$
\Dfin(\nu) : = \mathrm{Span}\{ \chi_A : A \in \Bfin(\nu)\}.
$$
It is straightforward to  check that Assumption 1 implies 
$$
\Dfin(\mu) \subset \Dfin(\nu).
$$
In general, the converse does not hold. But these two sets coincide if and
 only if Assumption 1 is extended by adding  the reverse implication
$$
\int_A c(x) \; d\mu(x) \ \Longrightarrow \ \mu(A) < \infty.
$$
\end{remark}

\subsection{Embedding operator J}\label{subsect J}
We define now a natural embedding $J$ of bounded Borel functions over 
$(V, \B)$ into bounded Borel functions over $\vv$. The operator $J$ will
 be considered later acting on the corresponding $L^2$-spaces.
 
 Let
\be\label{eq embedd J}
(Jf)(x, y) = f(x), \ \ \  \ f \in \FVB.
\ee
If $\VB$ is equipped with a $\sigma$-finite measure $\mu$ (or 
$\nu = c\mu$), we can specify
$J$ as an operator with domain $L^2(\mu)$ or $L^2(\nu)$). 

\begin{theorem}\label{prop embedd J}
For given $\sms$, let $\rho$  be a symmetric measure $\rho$ on $\vv$
and $c(x) = \rho_x(V)$.  
Then: 

(1) the operator $J : L^2(\nu) \to L^2(\rho)$ is an isometry where
$d\nu(x) = c(x) d\mu(x)$;

(2) the co-isometry $J^* : L^2(\rho) \to L^2(\nu)$ acts by the formula
$$
(J^*g)(x) = \int_V g(x,y)\; P(x, dy), \qquad g\in L^2(\rho);
$$

(3) the operator $J : L^2(\mu) \to L^2(\rho)$ is densely defined (in 
$L^2(\mu)$) and is, in general, unbounded. 

\end{theorem}

\begin{proof}
(1) This fact is proved by the following computation: for any 
$f \in L^2(\nu)$, one has
 $$
 \ba 
 || (Jf) ||^2_{L^2(\rho)} = & \iint_{V\times V} (Jf)^2(x, y) \; d\rho(x, y)\\
 =&  \iint_{V\times V} f^2(x) \; d\rho_x (y)d\mu(x)\\
 =&   \int_{V} f^2(x) c(x)\; d\mu(x)\\
 =&  || f ||^2_{L^2(\nu)}.
 \ea
 $$
 
 (2) To find the co-isometry $J^*$, we take arbitrary functions $f \in
  L^2(\nu)$ and $g\in L^2(\rho)$ and compute the inner product 
  using the  equality $c(x) P(x, dy) = d\rho_x(y)$:
 $$
 \ba 
 \langle Jf, g\rangle_{L^2(\rho)} = & \iint_{\VtV} (Jf)(x, y) g(x, y)\;
 d\rho(x, y)\\
= & \int_{V} f(x) \left(\int_{V}  g(x, y)\; d\rho_x(y)\right) d\mu(x)\\
= & \int_{V} f(x) \left(\int_{V}  g(x, y)\; P(x, dy)\right) d\nu(x)\\
= &\langle f, J^*g\rangle_{L^2(\nu)},
 \ea
 $$
where $J^*g =\int_{V}  g(x, y)\; P(x, dy)$.  This proves (2).
 
 (3) To show that (3) holds, we take a Borel function $f \in L^2(\mu)$ and 
 note that 
 \be\label{eq_J in mu-space}
 || Jf ||^2_{L^2(\rho)} = \iint_{\VtV} f^2(x) \; d\rho_xd\mu(x) = 
  \int_V f^2(x)c(x) \; d\mu(x).
 \ee
 In particular,  we have, for $A\in \Bfin$,
 $$
  || J(\chi_A) ||^2_{L^2(\rho)} = \int_A c(x) \; d\mu(x),
 $$
 that is, assuming that $c$ is locally integrable, we see that 
 $J$ is well defined on a dense subset of $L^2(\mu)$.
 Formula (\ref{eq_J in mu-space}) shows that,  for general $c$, the operator
 $J : L^2(\mu) \to L^2(\rho)$ is not  bounded.   
 \end{proof}

\section{Equivalence of  symmetric measures}\label{sect equivalence}
In this section we focus on the question about relations of Markov 
operators, and Laplacians, arising from equivalent symmetric measures.

\subsection{Equivalence of Markov operators}
Let $\rho$ be a symmetric measure on $\vv$ which is disintegrated
by fiber measures $x \mapsto \rho_x$ over the measure $\mu = 
\rho\circ \pi^{-1}$.  As above, define transition probabilities  $x \mapsto 
 P(x, \cdot)$ by setting $c(x)^{-1} d\rho_x(\cdot) = P(x, \cdot)$ where 
 $c(x) = \rho_x(V)$. In other words, $P(x, A) = P(\chi_A)(x)$ where 
 $P$ is the Markov operator, see \eqref{eq formula for P}. 

Having the operator $P$ defined, one can construct a stationary Markov 
process. Let $\Omega = V \times V \times V \times \cdots = 
V^{\mathbb  N_0}$. For $\omega = (\omega_n)\in \Omega$, set 
$$
X_n : \Omega \to V : X_n(\omega) = \omega_n, \qquad n\in \N_0.
$$

These notions are studied in detail in Section \ref{sect Transient}. 
Here we mention only the notion of \textit{reversibility},  one of the most
 important properties of Markov operators (processes). 
 
 \begin{definition}\label{def reversible MP} 
(1) A kernel $x \mapsto k(x, \cdot)$ is called \textit{reversible} with 
respect to a measure $\mu$ on $\VB$, if for any bounded Borel function
 $f(x,y)$,
 $$
 \iint_{\VtV} f(x, y) k(x, dy)d\mu(x) = \iint_{\VtV} f(y, x) 
 k(x, dy)d\mu(x). 
 $$
 
(2) Suppose that $x \mapsto P(x, \cdot )$ is a measurable family of
 transition probabilities on the space $\sms$, and let $P$ be the  Markov 
 operator determined by $x \mapsto P(x, \cdot )$. It is said that the 
 corresponding Markov process  is \textit{reversible} with respect to a 
measurable functions $c: V \to  (0, \infty)$ if, for any sets $A, B \in \B$,
  the following relation holds:
\be\label{eq def reversible P}
 \int_B c(x) P(x, A)\; d\mu(x) = \int_A c(x) P(x, B)\; d\mu(x).
\ee

Denoting $d\nu(x) = c(x)d\mu(x)$, we can rewrite 
\eqref{eq def reversible P} in the form that will be used below.
$$
\int_V \chi_B(x) P(x, A) \; d\nu(x) = \int_V \chi_A(x) P(x, B) \; 
d\nu(x).
$$
\end{definition} 

The following result clarifies relationship between symmetric measures 
$\rho$ and reversible Markov processes. This lemma is a part of
more general statement, see Theorem \ref{prop_reversible}.

\begin{lemma}\label{lem sym vs revers}
Let $\rho= \int_V\rho_x \; d\mu$ be a measure on $\vv$ such that 
$c(x) = \rho_x(V) <\infty$. Suppose that the Markov operator $P$
is defined according to \eqref{eq formula for P}. Then the following
are equivalent:

(i) $\rho$ is symmetric;

(ii) $ (P,c)$ is reversible.
\end{lemma}

In what follows, we will focus on the following \textit{question}: 
suppose that $\rho $
and $\rho'$  are two \textit{equivalent} symmetric measures such that 
the corresponding Markov processes $(P,c)$ and $(P',c')$ are reversible.
How are they related? More generally, we can ask about relations between 
all objects whose definition was based on a symmetric measure. They are
the Laplacian $\Delta$, symmetric operator $R$, and finite energy
Hilbert space. Some partial answers are given in this and  subsequent 
sections.

\begin{definition}\label{def_MP equivalence}
Let $(P, c)$ be a pair consisting of a positive measurable function $c(x)$
on $\sms$ and a reversible Markov process $P(x, \cdot)$ satisfying 
Definition \ref{def reversible MP}. We will say that two such pairs
$(P, c)$ and $(P',c')$ are \textit{equivalent} if the corresponding 
symmetric 
measures $\rho$ and $\rho'$ are equivalent as measures on $\vv$ (see
Theorem \ref{prop_reversible}). 
The latter  means that there exists a positive measurable function $r(x, y)$
 such that
$$
d\rho'(x, y) = r(x, y) d\rho(x, y).
$$

If the equivalent measures $\rho$ and $\rho'$ satisfy the property 
$\mu = \rho\circ\pi_1^{-1}
= \rho'\circ\pi_1^{-1}$, then we call the pairs $(P, c)$ and $(P',c')$
\textit{strongly equivalent.} In this case, we also call the  measures
$\rho$ and $\rho'$ \textit{strongly equivalent}.
\end{definition}

\begin{remark}\label{rem equiv measures}
 (1) The symmetry of equivalent measures $\rho$ and 
$\rho'$ implies that the function $r(x, y)$ is symmetric, $r(x, y) = r(y,x)$.

(2) Let the measures $\rho$ and $\rho'$  be strongly equivalent.
Then these measures are disintegrated as follows:
$$
\rho' = \int_V \rho'_x\ d\mu(x), \qquad \rho = \int_V \rho_x\ d\mu(x).
$$
It can be seen that the equivalence of $\rho$ and $\rho'$ implies
that the measures $\rho_x$ and $\rho'_x$ are equivalent $\mu$-a.e.
Moreover, 
\be\label{eq_RN for every x}
\frac{d\rho'_x}{d\rho_x}(y) = r_x(y)
\ee
where $r_x(\cdot)$ is obtained from $r(x, \cdot)$ by fixing the variable
 $x$. 

(3) Conversely, given two  (strongly) equivalent measures $\rho$ and 
$\rho'$, we can construct  (strongly) equivalent pairs $(P,c)$ and 
$(P',c')$ according to the properties formulated in Lemma \ref{lem sym vs
revers} and Theorem
\ref{prop_reversible}. In other words, if $(P,c)$ defines a  reversible 
Markov process with the symmetric measure $\rho$, then, for any 
symmetric measure $\rho'$ equivalent to $\rho$, we can construct a 
reversible Markov process $(P', c')$ which is equivalent to $(P,c)$. 
Note that the functions $c(x) = \rho_x(V)$ and 
$c'(x) = \rho'_x(V)$ are determined by $\rho$ and $\rho'$ uniquely.
\end{remark}

One can prove a more general statement than that given in Remark 
\ref{rem equiv measures} (2).

\begin{lemma} Let $\rho$ and $\rho'$ be two symmetric measures on 
$\vv$ such that $d\rho'(x, y) = r(x, y) d\rho(x, y)$. Suppose that 
$$
\rho' = \int_V \rho'_x\ d\mu'(x), \qquad \rho = \int_V \rho_x\ d\mu(x)
$$
and the measures $\mu$ and $\mu'$ on $\VB$ are equivalent,
i.e., $m(x) d\mu'(x) =  d\mu(x)$ for some positive Borel function $m(x)$.
Then the measures $\rho'_x$ and $\rho_x$ are equivalent a.e. on $V$, and
\be\label{eq-RN rho and rho'}
\frac{d\rho'_x}{d\rho_x}(y) = m(x) r_x(y).
\ee
\end{lemma}

\begin{proof} (Sketch) The result is deduced as follows:
$$
\ba
\rho'(A\times B) = &\ \iint_{A\times B} r(x, y) \; d\rho(x,y)\\
 = & \ \iint_{A\times B} r(x, y) \; d\rho_x(y) d\mu(x)\\
 = &\ \int_{A} \left( \int_B m(x) r(x, y) \; d\rho_x(y) \right) d\mu'(x).
\ea
$$
On the other hand, 
$$
\rho'(A\times B) = \int_A \rho'_x(B) \; d\mu'(x).
$$
Comparing the above formulas, we obtain that (\ref{eq-RN rho and rho'})
holds.

\end{proof}

Consider a particular case when the Radon-Nikodym derivative $r(x, y)$ of 
two equivalent measures $\rho$ and $\rho'$ is the product $p(x)q(y)$.

\begin{lemma}\label{lem r(x,y) = r(x)r(y)} Let $\rho = \int \rho_x \; 
d\mu(x)$ and $\rho' = \int \rho'_x \; d\mu'(x)$ be two measures on $\vv$
such that 
$$
\frac{d\rho'}{d\rho}(x, y) = p(x)q(y)
$$
for some  positive Borel functions $p$ and $q$. Then, for $\mu$-a.e. 
$x \in V$, the Radon-Nikodym derivative 
$\dfrac{d\rho'_x(y)}{d\rho_x(y)}$ satisfies the relation
\be\label{eq_product RN der}
\frac{1}{q(y)}\frac{d\rho'_x(y)}{d\rho_x(y)} = \varphi(x)
\ee
where 
$$
\varphi(x) = p(x) \frac{d\mu}{d\mu'}(x).
$$
\end{lemma}

\begin{proof}
The result can be easily deduced from  the formula 
$$
d\rho'_x(y)d\mu'(x) = p(x) q(y) d\rho_x(y)d\mu(x). 
$$
We  leave the details to the reader.
\end{proof}

Relation (\ref{eq_product RN der}) means that the Radon-Nikodym 
derivative $\dfrac{d\rho'_x}{d\rho_x}(y)$ is proportional to the function
$q(y)$ where the coefficient of proportionality is given by $\varphi(x)$. 
If $\rho$ and $\rho'$ are symmetric measures, then
$\dfrac{d\rho'}{d\rho}(x, y) = p(x)p(y)$.

\begin{theorem}\label{lem_ P eq P'} Let $\rho$ and $\rho'$ be two 
strongly equivalent  measures on $\vv$ such that $d\rho'_x =
r_x(y) d\rho_x(y)$ for all $x\in V$.  Then  the corresponding Markov
 processes   $(P,c)$ and $(P', c')$ are strongly equivalent and 
\be\label{eq P' and P}
P'(f)(x) = \frac{P(fr_x)(x)}{P(r_x)(x)}.
\ee
\end{theorem}

\begin{proof} We first find $P(r_x)$:
\begin{eqnarray} \label{eq P(r_x)}\nonumber
P(r_x)(x)  &= & \int_V \frac{d\rho'_x}{d\rho_x}(y) \; P(x, dy)\\
\nonumber
&= & \frac{1}{c(x)}  \int_V \frac{d\rho'_x}{d\rho_x}(y) \; d\rho_x(y)\\
&= & \frac{1}{c(x)}  \int_V \; d\rho'_x(y)\\
\nonumber
&= & \frac{c'(x)}{c(x)}.
\nonumber
\end{eqnarray} 

Next, we compute 
$$
\ba 
P'(f)(x) =& \int_V f(y) \; P'(x, dy)\\
= & \frac{1}{c'(x)} \int_V f(y)  \; d\rho'_x(y)\\
= & \frac{1}{c'(x)} \int_V f(y) r_x(y) \; d\rho_x(y)\\
= & \frac{c(x)}{c'(x)} \int_V f(y) r_x(y) \; dP(x, dy)\\
= & \frac{c(x)}{c'(x)} P(f r_x)(x)\\
\ea
$$
Now, the result follows from (\ref{eq P(r_x)}).
\end{proof}

\begin{remark}
(1) Let the symmetric measures $\rho$ and $\rho'$ be strongly equivalent,
$d\rho'_x(y)= r_x(y) d\rho_x(y)$. As in (\ref{eq P(r_x)}), we can 
obtain that 
$$
P'\left(\frac{1}{r_x}\right)(x) = \frac{c(x)}{c'(x)}.
$$ 
Therefore, the following property holds:
$$
P(r_x)(x) P'\left(\frac{1}{r_x}\right)(x) = 1
$$

(2) Since the notion of equivalence of measures $\rho$ and $\rho'$ is
symmetric, we note that the roles of $P$ and $P'$ can be 
interchanged and the following relation holds:
$$
P(f)(x) = \frac{P'\left( f \frac{1}{r_x} \right) (x)}  {P'\left(\frac{1}{r_x}
 \right)(x)}.
$$

(3) It follows from the strong equivalence of $\rho$ and $\rho'$ that 
$r_x(y)$ is integrable with respect to $\rho_x$ and  
$$
c'(x) = \int_V r_x(y) \; d\rho_x(y). 
$$

(4) Several useful formulas can be easily deduced from  Theorem
\ref{lem_ P eq P'}. Firstly, formula \eqref{eq P' and P} can be rewritten
in the form
\be\label{eq_ P via c, c', and P'}
P(f r_x)(x) = c'(x) P'(f)(x) c(x)^{-1},
\ee
and equivalently, the latter is represented as a relation between Markov
kernels:
$$
c'(x) P'(x, dy) = c(x) r_x(y) P(x, dy).
$$

(5) The same proof as in Theorem \ref{lem_ P eq P'} shows that 
$$
R'(f)(x) = R(fr_x)(x).
$$

(6) In more general setting, assuming that $d\rho'_x(y) = m(x) r_x(y) 
d\rho_x(y)$ where $m(x)$ is as in  (\ref{eq-RN rho and rho'}), we deduce
that 
$$
P(f r_x)(x) m(x)= c'(x) P'(f)(x) c(x)^{-1}.
$$

Similarly, one can show that 
$$
R'(f)(x) = m(x) R(f r_x)(x)
$$
where the operator $R'$ is defined by $x \mapsto \rho'_x$. 

(7) Suppose that, for given pair $(P, c)$, the operator $P'$ is defined by
(\ref{eq_ P via c, c', and P'}), and let $d\nu'(x) = c'(x) d\mu(x)$. Then
we claim that $\nu' P' = \nu'$:

$$
\ba
\int_V P'(f)(x)\; d\nu'(x) = & \ \int_V c(x) P(f r_x)(x) c'(x)^{-1} c'x) \; 
d\mu(x) \\
= & \ \int_V P(fr_x)(x) \; d\nu(x)\\
= &\ \int_V\left( \int_V (fr_x)(y) P(x, dy) \right)\; d\nu(x)\\
= & \ \iint_{\VtV} f(y) \frac{d\rho'_x}{d\rho_x}(y) c(x)^{-1} d\rho_x(y)
c(x) d\mu(x)\\
= & \ \iint_{\VtV} f(y) \; d\rho'_x(y) d\mu(x) \\
= & \ \iint_{\VtV} f(x) \; d\rho'(x,y)\\
= & \int_V f(x) c'(x)\; d\mu(x)\\
= & \ \int f(x) \; d\nu'(x).
\ea
$$
\end{remark}

\subsection{On the Laplacians $\Delta$ and $\Delta'$} 
In the remaining part of this section, we will discuss relations between 
the Laplace operators $\Delta$ and $\Delta'$ acting in the finite
energy Hilbert spaces $\h_E(\rho)$ and $\h_E(\rho')$ respectively. 

Let $\Delta'(f)$ be the Laplace operator defined by a symmetric measure
$\rho'$ on $\vv$. We can find out how $\Delta'$ and $\Delta$ are
 related.

\begin{proposition}\label{lem_P' preserves nu'}
Let $\rho$ and $\rho'$ be two equivalent symmetric measures
on $\vv$ such that $d\rho'(x, y) = q(x)q(y)d \rho(x, y)$. Then 
$$
\Delta'(f) =  cqf(P(q) - q) + q \Delta(qf).
$$
In particular, when $q$ is harmonic for $P$, then 
\be\label{eq Delta' via Delta}
\Delta'(f) = q \Delta (qf).
\ee
Moreover, 
$$
\Delta'(f) =0 \ \Longleftrightarrow \ P(qf) = f P(q),
$$
and assuming that $P(q) = q$, we have 
$$
f \in \h arm(\Delta') \ \Longleftrightarrow \ qf \in \h arm (\Delta).
$$ 
\end{proposition}

\begin{proof}
(1) By definition of the operator $\Delta$, we have
\be\label{eq Delta'}
\ba
\Delta'(f)(x) = & \ \int_V (f(x) - f(y)) \; d\rho'_x(y) \\
= &\  \int_V (f(x) - f(y))q(x) q(y) \; d\rho_x(y) \\
= &\  \int_V (f(x) - f(y)) c(x)q(x) q(y) \; dP(x, dy) \\
= & \ c(x) q(x) f(x) \int_V q(y) \; P(x, dy) - c(x) q(x) \int_V q(y) f(y)\;
P(x, dy)\\
= & \ c(x)q(x) \left[ f(x) P(q)(x) - P(qf)(x) \right]. 
\ea
\ee
Add and subtract $cq^2f$ to the right hand side of (\ref{eq Delta'}). Then, 
regrouping the terms, we obtain
$$
\Delta'(f) = c q [qf - P(qf)] + cqf (P(q) - q) = q \Delta (qf) + cqf (P(q) - q). 
$$
This means that, in case when $P(q) = q$, the Laplace operators 
$\Delta$ and $\Delta'$ are related as in (\ref{eq Delta' via Delta}).

(2) Now we can apply (1) to prove the formulas given in (2). 
From the last expression in (\ref{eq Delta'}), we see that $f$ is harmonic
with respect to $\Delta'$ if and only if $P(qf) = fP(q)$. 
\end{proof}

\begin{corollary}
Let $\rho$ be a symmetric measure on $\vv$, and let $q$ be a 
 harmonic
function for the Markov operator $P$ generated by $\rho$. Define
the symmetric measure $\rho'$ such that $d\rho'(x,y) = q(x)q(y)
d\rho(x,y)$. Let $P'$ be the corresponding Markov operator produced by
$\rho'$. Then we have the map
$$
\h arm(P') \times \h arm(P) \ni (f, q) \mapsto fq \in \h arm(P).
$$
\end{corollary}

\begin{proof}
It follows from the definition of the measure $\rho'$ that
$$
c'(x) = \int_V\; d\rho'_x(y) = \int_V q(x) q(y)\; d\rho_x(y) = 
q(x) R(q)(x).
$$
Since $q$ is harmonic, i.e., $R(q) = cq$, we obtain that 
\be\label{eq_c c' q}
  c'(x) =c(x) q^2(x).
 \ee

Let $f $ be any function harmonic with respect to the operator $P'$. 
Then 

$$
\ba 
f(x) & = \int_V f(y) \; P'(x, dy)\\
& = \frac{1}{c'(x)} \int_V f(y) \; d\rho'_x(y)\\
& = \frac{q(x)}{c'(x)} \int_V f(y)q(y) d\rho_x(y)\\
& = \frac{q(x)}{c'(x)} \int_V f(y)q(y) c(x) P(x, dy)\\
& = \frac{q(x)c(x)}{c'(x)} P(qf)(x)\\
\ea
$$
It follows from \eqref{eq_c c' q} that $f = q^{-1} P(qf)$, and we are
done.
\end{proof}

We remark that in the proved statement we temporarily 
 extended the notion of symmetric measures to the case of 
\textit{signed symmetric measures} assuming that
the $P$-harmonic function $q$ can be negative. 

\begin{theorem}\label{thm energy rho and rho'}
Suppose that $\rho'$ and $\rho$ are two symmetric measures such that
$d\rho'(x, y) = q(x)q(y)d \rho(x, y)$. If $q$ is harmonic for the Laplace 
operator $\Delta$, then the operator
$$
Q : \h_E(\rho') \to \h_E(\rho)  : Q(f) = qf
$$
is an isometry. 
\end{theorem}

\begin{proof} 
We need to show that, for any $f \in \h_E(\rho')$,
$$
|| f ||_{\h_E(\rho')} = ||q f ||_{\h_E(\rho)}.
$$
 In the  computation given below, we use the following:
 the definition of the norm in the
 finite energy space, the symmetry of the measures $\rho$ and 
 $\rho'$, and the relation $R(q) = cq$ that holds for harmonic functions
 because 
 $$
\Delta (q)(x ) = c(x) q(x) - R(q)(x).
$$
 Then we compute

$$
\ba
|| f ||^2_{\h_E(\rho')} - ||q f ||^2_{\h_E(\rho)} = &\ \frac{1}{2} 
\iint_{\VtV}  (f(x) - f(y))^2 \; d\rho'(x,y) \\ 
 & \ \ \ \qquad - \iint_{\VtV}  (q(x)f(x) - q(y)f(y))^2 \; d\rho(x,y)\\
 = & \iint_{\VtV}  [ (f(x) - f(y))^2 q(x)q(y) \\ 
 & \qquad \quad - (q(x)f(x) - q(y)f(y))^2] \; d\rho(x,y)\\
= & \ \iint_{\VtV} [f^2(x)q(x) q(y) - q^2(x) f^2(x) ] \\
& \qquad \quad + [f^2(y) q(x)q(y) - q^2(y) f^2(y) ] \; d\rho(x,y)\\
= &\ 2 \ \iint_{\VtV} [f^2(x)q(x) q(y) - q^2(x) f^2(x) ] \; 
d\rho_x(y)d\mu(x)\\
= &\ 2\ \int_V f^2(x)q(x) [R(q)(x) - c(x) q(x)]\; d\mu(x)\\
= & \ 0.
\ea
$$
This computation shows that $Q(f) = qf \in \h_E(\rho)$ and $Q$ 
preserves the norm. 
\end{proof}

Continuing the above theme, consider the Laplace operator $\Delta$
acting in $L^2(\mu)$. We recall that $\Delta : L^2(\mu) \to L^2(\mu)$
is a positive definite self-adjoint operator according to 
Proposition \ref{prop prop of R, P, Delta}.   

\begin{proposition}\label{prop Delta(qf)} Suppose $\rho$ is a symmetric measure on 
$\vv$ and the Laplacian $\Delta= \Delta(\rho)$ is defined by
\eqref{eq def of Delta}. 
Let $q$ and $f$ be functions on $\sms$ from the domain of $\Delta$
such that $qf$ is also in the domain of $\Delta$.
Then
\be\label{eq_Delta(qf)}
\int_V \Delta(qf)\; d\mu = \int_V q \Delta(f) \; d\mu - 
\int_V f \Delta(q)\; d\mu.
\ee
If $q$ and $f$ are in $L^2(\mu)$, then $\int_V \Delta(qf) \; d\mu =0$.
\end{proposition}

\begin{proof} By definition of $\Delta$, we have
$$
\ba
\Delta(qf) & = \int_V [ (qf(x) - qf(y)]\; d\rho_x(y)\\
& = \int_V (q(x) f(x) - q(x) f(y) + q(x)f(y) - q(y) f(y))\; d\rho_x(y)\\
& = q(x) \Delta(f) - \int_V f(y) (q(x) - q(y))\; d\rho_x(y)\\
\ea
$$
Then

$$
\ba
\int_V \Delta(qf)(x)\; d\mu(x) & = \int_V q\Delta(f) \; d\mu(x) + 
\iint_{\VtV}  f(y) (q(x) - q(y))\; d\rho_x(y) d\mu(x)\\
& = \int_V q\Delta(f) \; d\mu(x) + 
\iint_{\VtV}  f(x) (q(y) - q(x))\; d\rho_x(y) d\mu(x)\\
& = \int_V q\Delta(f) \; d\mu(x) - 
\int_{V}  f \Delta(q)\;  d\mu(x)\\
\ea
$$
and \eqref{eq_Delta(qf)} is proved. 

If the functions $q$ and $f$ are in $L^2(\mu)$ (in particular, $q$ and
$f$ can be taken from the dense subset $\Dfin(\mu)$), then 
we can use the fact that $\Delta$ is essentially self-adjoint and conclude
that
$$
\int_V \Delta(qf)(x)\; d\mu(x) = \langle q, \Delta(f) \rangle_{L^2(\mu)}
 - \langle \Delta (q), f \rangle_{L^2(\mu)} = 0.
$$ 
\end{proof}
 
We immediately deduce the following fact from Proposition 
\ref{prop Delta(qf)}.

\begin{corollary}
(1)  If functions  $f$ and $f^2$ are in the domain of $\Delta$, then
 $$\int_V \Delta(f^2)\; d\mu =0.
 $$
 
 (2) If $f$ is a harmonic function for $\Delta$, then $\Delta(f^2) \leq 0$, 
 and therefore $f^2$ is also harmonic. 
 \end{corollary}

\begin{proof} (1) is obvious. To show that (2) holds, we use that 
$\Delta(f) = c(f - P(f)$ and $P$ is a positive operator. This means that
$P(f) \geq 0$ whenever $f \geq 0$. By Schwarz' inequality for 
positive operators, we have $P(f^2)(x) \geq P(f)^2(x)$, and therefore 
$$
\ba
\Delta(f^2) & = c(f^2 - P(f^2))\\
& \leq c(f^2 - P(f)^2)\\
&= c (f - P(f))(f+P(f))\\
&=0.
\ea
$$
The fact that $f^2$ is harmonic follows from (1) and the proved 
inequality in (2). 
\end{proof}


\section{Reversible Markov process generated by symmetric measures}
\label{sect Transient}

In this section, we consider Markov processes generated by a Markov
operator which is determined by a symmetric irreducible measures 
$\rho$ on the standard Borel space $\vv$ such that 
the margin measure $\mu$ on $\VB$ is \textit{$\sigma$-finite}. 
We will assume that this Markov process is transient (see the definition
 below).
The reader can find vast literature on the theory of transient Markov 
processes, we refer to \cite{Artalejo2011, Cyr2010, Korshunov2008, 
LyonsPeres2016, Nummelin1984,  Revuz1984, WeiKryscio2016,
 Woess2009}.

\subsection{Reversible Markov processes}
 
Let $\sms$ be a $\sigma$-finite measure space, and let $\rho$ be a 
symmetric measure on $\vv$ which is disintegrated  with respect to
 $(\rho_x, x \in V)$ and $\mu$ according to 
(\ref{eq disint for rho}). By assumption, $c(x) = \rho_x(V)$ is locally 
integrable. We recall (see Definition \ref{def R, P, Delta})
 that, in this setting,  a Markov operator $P$ is defined 
on $\FVB$ by the probability  kernel $x \mapsto P(x, \cdot)$. This 
operator $P$ acts by the formula
\be\label{eq-P via P(x,dy)}
P(f)(x)  = \int_V f(y) \; P(x, dy)
\ee
where $P(x, dy) =  c(x)^{-1}d\rho_x(y)$. Then the operator $P$ is 
positive and normalized, i.e., $P(\mathbbm 1) = \mathbbm 1$. As 
mentioned above in Proposition \ref{prop prop of R, P, Delta}, the fact that
$\rho$ is symmetric is equivalent to 
 self-adjointness of $P$ as an operator in $L^2(\nu)$. It follows also that 
 $P$ preserves  the measure $\nu = c \mu$.
Furthermore, we can use the kernel $x \to P(x, \cdot) = P_1(x, \cdot)$ to
define the sequence of probability kernels (transition probabilities)
$(P_n(x, \cdot) : n \in \N)$ in accordance with (\ref{eq_powers of k}). 
These kernels satisfy the equality 
$$
 P_{n+m}(x, A)  = \int_V P_n(y, A)  P_m(x, dy), \qquad n, m \in \N.
$$
Therefore one has
$$
P^n(f)(x)  = \int_V f(y) \; P_n(x, dy), \qquad n \in \N,
$$ 
and this relation defines the sequence of probability measures 
$(P_n)$ by setting $P_0(x, A) = \delta_A(x) = \chi_A(x)$ and
$$
P_n(x, A) = P^n(\chi_A) = \int_V \chi_A(y) \; P_n(x, dy), 
\qquad A\in \B, n\in \N.
$$
We use the notation $P(x, A)$ for $P_1(x, A)$. 

For the Markov operator $P$, one can define one more sequence 
of measures. We use the formula
\be\label{eq_def rho_n}
\rho_n(A \times B) = \langle \chi_A, P^n(\chi_B)\rangle_{L^2(\nu)}, 
\ee
to define the measures $\rho_n, n \in \N,$ on the Borel space $\vv$
(here  $\rho_1 = \rho$).

\begin{lemma} 
(1) Every measure $\rho_n, n \in \N,$ is symmetric on $\vv$, and
$\rho_n $ is equivalent to $\rho$.

(2) $\rho_x^{(n)}(V) = c(x), \forall n\in \N$.

(3) 
\be\label{eq_rho_n via P_N}
d\rho_n(x,y) = c(x) P_n(x, dy)d\mu(x)= P_n(x, dy)d\nu(x).
\ee

(4) 
$$
\rho_n( A\times B) = \langle \chi_A, RP^{n-1}(\chi_B)
\rangle_{L^2(\mu)}.
$$
\end{lemma}

\begin{proof}
The assertions of the lemma are rather obvious. We only mention two 
simple facts: $\rho_n(A \times V) = \rho(A \times V)$ for every $n$, and,
 since the operator $P^n$ is self-adjoint in $L^2(\nu)$, the measure
$\rho_n$ is symmetric. 
\end{proof}

\begin{definition}\label{def reversible MP-1} 
Suppose that $x \mapsto P(x, \cdot )$ is a measurable family of transition
 probabilities on the space $\sms$, and let $P$ be the  Markov operator 
determined by $x \mapsto P(x, \cdot )$. It is said that the corresponding
 Markov process  is  \textit{reversible} with respect to a measurable 
 functions $c: x \to  (0, \infty)$ on $\VB$ if, for any sets $A, B \in \B$, the
following relation holds:
\be\label{eq def reversible P}
 \int_B c(x) P(x, A)\; d\mu(x) = \int_A c(x) P(x, B)\; d\mu(x).
\ee
\end{definition} 

As shown in \cite{BezuglyiJorgensen2018}, the reversibility 
for the Markov process $(P_n)$ is equivalent to the following properties
 (here we give an extended and  more comprehensive formulation):

\begin{theorem}\label{prop_reversible}
Let $\sms$ be a standard $\sigma$-finite measure space, $x \mapsto c(x)
\in (0, \infty)$ a measurable function, $c \in \Lloc$. Suppose that 
$x \mapsto P(x, \cdot )$ is a probability kernel. 
The following are equivalent:

(i) $x \mapsto P(x, \cdot )$ is reversible (i.e., it satisfies 
(\ref{eq def reversible P}); 

(ii) $x\to P_n(x, \cdot)$ is reversible for any $n\geq 1$;

(iii) the Markov operator $P$ defined by $x\to P(x, \cdot)$ is  self-adjoint 
on $L^2(\nu)$ and $\nu P= \nu$ where $d\nu(x) = c(x) d\mu(x)$;

(iv) 
$$
c(x) P(x, dy) d\mu(x) = c(y) P(y, dx)d\mu(y);
$$

(v) the operator $R$ defined by the relation $R(f)(x) = c(x)P(f)(x)$
is symmetric (see Remark \ref{rem_on operators});

(vi) the measure $\rho$ on $(V\times V, \B\times \B)$ defined by 
$$
\rho(A \times B) = \int_V \chi_A R(\chi_B)\; d\mu = 
\int_V c(x) \chi_A P(\chi_B)\; d\mu
$$
is symmetric;

(vii) for every  $n \in \N$, the measure $\rho_n$ defined by 
(\ref{eq_def rho_n}) is symmetric.
\end{theorem}

We discuss the notion of reversibility in the following Remark where we
included several direct consequences of Definition \ref{def reversible MP}
and Theorem \ref{prop_reversible}. 

\begin{remark} (1) Let $x \mapsto P(x, \cdot) $ be a Borel field of 
probability measures over a standard Borel space $(V, \B)$. This field of
transition probabilities generates the Markov operator $P$ such that
$P(\mathbbm 1) = 1$. 
It follows from
Theorem \ref{prop_reversible} that one can define the notion 
of reversible Markov process $x \mapsto P(x, \cdot) $ with respect to a 
$\sigma$-finite measure $\nu$: It is said that $((x \mapsto P(x, \cdot)),
 \nu) $ is \textit{reversible} if $P$ is a self-adjoint operator in 
 $L^2(\nu)$. This definition is equivalent to the property
 $$
 \int_A P(x, B) \; d\nu = \int_B P(x, A)\; d\nu.
 $$
Equally, one can consider the notion of reversibility for 
$P(x, \cdot)$ with respect to a symmetric  measure $\rho$. Theorem 
\ref{prop_reversible} states the equivalence of these approaches.

(2) Based on (1), the following \textit{question} is raised naturally: 
\textit{Given $x \mapsto P(x, \cdot)$ as  above, under what condition  
the set 
$$
\mathcal S(P) := \{ \nu : P \ \mbox{is\ self-adjoint\ in}\ L^2(\nu)\}
$$
is non-empty?}

(3) The following observation is a direct consequence of 
Theorem \ref{prop_reversible}. Let $P(x, A) =  P(\chi_A)(x)$ be the
 probability kernel  defined by a normalized Markov operator $P$ 
 acting on Borel functions over $\sms$. To answer the question about the 
 existence of a $P$-invariant measure $\nu \sim \mu$ such that
 $(P, \nu)$ is reversible, it suffices  to construct a locally integrable
  function $c$ satisfying (\ref{eq def reversible P}).  It can be done by
 pointing out a symmetric measure $\rho $ such that $\rho_x(V) = c(x)$ 
 and the projection of  $\rho$ onto $V$ is the measure $\mu$.
 
 (4) There exists a stronger version of reversible Markov processes. 
 Let $P$ be a Markov operator acting on $\FVB$ such that, for any
 $A, B \in \Bfin(\mu)$, 
 $$
 \chi_A P(\chi_B) = \chi_B P(\chi_A).
 $$ 
 Then, for any positive Borel function
$c \in \Lloc$, the measure $d\nu(x) = c(x) d\mu(x) $ belongs to
$\mc S(P)$. 
Indeed, it suffices to define  the symmetric measure $\rho$ according to 
Theorem \ref{prop_reversible} (vi) and then apply statement (ii).

(5) We give here one more interpretation of the definition of 
reversible Markov processes. For this, we use notation introduced  in 
Section \ref{subsect Path space MP}. Let 
$$
\Omega = V \times V \times V \cdots$$
 be the path space of the Markov
process $(P_n)$, and let $X_n : \Omega \to V$ be the random variable 
defined by $X_n(\omega) = \omega_n$. Given a measure $\nu$ on $V$,
we can reformulate the definition of reversible Markov operator as follows:
$$
dist( X_0\ |\ X_1 \in A) = dist(X_1\ |\ X_0 \in A).
$$
The meaning of the above formula is clarified in Proposition 
\ref{lem symm distr}.

(6) Suppose now that  a non-symmetric measure $\rho$ is given on the
space $\vv$, i.e, $\rho(A \times B) \neq \rho(B\times A)$, in general. 
However, we will assume that $\rho$ is equivalent to $\rho\circ
\theta$ where $\theta(x, y) = (y, x)$. Then, using the same approach 
as above, we can define 
the following objects: margin measures 
$\mu_i := \rho\circ \pi_i^{-1}, i =1,2,$, fiber measures $d\rho_x(\cdot)$
and $d\rho^x(\cdot)$ (see Remark \ref{rem on symm meas}),  and 
functions $c_1(x) = \rho_x(V), c_2(x) = \rho^x(V)$.  

Define now the \textit{symmetric measure} $\rho^{\#}$ generated by 
$\rho$ as follows
$$\rho^{\#} := \dfrac{1}{2}(\rho + \rho\circ\theta).
$$
 Then
$$
\rho^{\#}(A \times B) =\frac{1}{2}(\rho(A\times B)+ \rho(B \times A)).
$$
Clearly, $\rho^{\#}$ is equivalent to $\rho$.

Let $E\subset \VtV$ be the support of $\rho$. Then $E^{\#} = E \cup 
\theta(E)$ is the  support of the symmetric measure  $\rho^{\#}$.  
The disintegration of $\rho 
= \int_V\rho_x\; d\mu_1(x)$ with respect to the partition 
$\{x\} \times E_x$ defines the 
disintegration of $\rho^{\#}$. For $\mu^{\#} := \dfrac{1}{2}(\mu_1 +
 \mu_2) $, we obtain that 
 $$
 \rho^{\#} = \int_V (\rho_x + \rho^x) \; d\mu^{\#}.
 $$

Having the symmetric measure $\rho^{\#}$ defined on $\vv$, 
we can introduce 
the operators $R^{\#}$ and $P^{\#}$ as in \eqref{eq def of R} and
\eqref{eq formula for P}. It turns out that, for $f\in \FVB$,
$$
R^{\#}(f)(x) = R_1(f)(x) + R_2(f)(x) 
$$
where
$$
R_1(f) = \int_V f(y) \; d\rho_x(y),\qquad R_2(f) = \int_V f(y) \; 
d\rho^x(y).
$$ 
Similarly, 
$$
P^{\#}(f)(x)  = \frac{1}{c^{\#}(x)} R^{\#}(f)(x)
$$
where
$$
c^{\#}(x) = \rho_x( V) + \rho^x(V).
$$ 
Then we can define the measure
$d\nu^{\#}(x) = c^{\#}(x) d\mu(x)$ such that the operator
$$
P^{\#} (f)(x) = \int_V f(y)  \frac{1}{c^{\#}(x)}\; d\rho^{\#}_x(y)
$$
is self-adjoint in $L^2(\nu^{\#})$. By Theorem \ref{prop_reversible},
we obtain that the Markov process generated by $x \mapsto 
P^{\#}(x, \cdot)$ is \textit{reversible} where $P^{\#}(x, A) = 
P^{\#}(\chi_A)(x)$. 

\end{remark}

\subsection{Properties of Markov operators} \label{subsect Properties MO}
In this subsection, we discuss  some properties of 
the Markov operator $P$, which is defined by relation 
(\ref{eq formula for P}). The operator $P$ is considered acting in Hilbert
 spaces  $L^2(\mu), L^1(\nu)$, and $\h_E$ where
$d\nu(x) = c(x)d\mu(x)$ and $\h_E$ is the energy space. 

We begin with the following simple observations whose proofs are obvious
and can be omitted.  Remind that 
$\Bfin(\mu)$ is the family  of Borel subsets of finite measure $\mu$, and 
$\Dfin = \Dfin(\mu)$ is the the linear subspace generated by the
 characteristic functions $\chi_A$, $A\in \Bfin$.

\begin{remark}\label{rem Bfin in spaces}
(1) If $c \in L^1_{\mathrm{loc}}(\mu)$, then 
$$
\Bfin(\mu) \subset \Bfin(\nu).
$$
The converse is not true. 

(2) We observe that if both functions, $c(x)$ and $c(x)^{-1}$ are in $\Lloc$,
then 
$$
\Bfin(\mu) = \Bfin(\nu).
$$

(3) The following property holds for  $c \in \Lloc$:
\be\label{eq_Dfin in 3 H sp}
\Dfin(\mu) \subset L^2(\mu) \cap L^2(\nu) \cap \h_E
\ee
(this should be understood that functions from $\Dfin$ are 
representatives of elements from $\h_E$).

(4) We recall that 
\be\label{eq norm chi_A in H}
\| \chi_A \|_{\h_E}^2 = \rho(A\times A^c) 
\ee
where $\rho$ is a symmetric measure used in the definition $\h_E$.
This fact is proved in \cite{BezuglyiJorgensen2018}.
\end{remark} 

\begin{lemma}
If $c \in \Lloc$, then $\Dfin(\mu)$ is dense in $L^1(\nu)$ and $L^2(\nu)$.
\end{lemma}

\begin{proof} (Sketch) We show the density of $\Dfin(\mu)$ in $L^1(\nu)$
only. It suffices to check that, for every $B \in \Bfin(\nu)$, the characteristic
function $\chi_B$ can be approximated in $L^1(\nu)$ by simple functions
from $\Dfin(\mu)$, i.e., for every $\varepsilon >0$, there exists some 
$s(x) \in \Dfin(\mu)$ such that $|| \chi_B - s ||_{L^1(\nu)} < \varepsilon$.
Without loss of generality, we can assume that $s(x) \leq \chi_B(x)$. Then
$$
|| \chi_B - s ||_{L^1(\nu)} = \int_V (\chi_B - s(x))\; d\nu(x)=
\int_B c(x) (1- s(x))\; d\mu(x).
$$
Since $c$ is $\mu$-integrable on $B$, one can take a subset $B_0 \subset
B$ such that 
$$
\int_B c \; d\mu - \int_{B_0} c \; d\mu < \varepsilon.
$$
Then result follows.
\end{proof}

Next, let $\rho$ be a symmetric measure on $\vv$, and let $P$ be the
 operator acting on bounded Borel functions by the formula
$$
P(f)(x) = \int_V f(y) P(x, dy)
$$
where $c(x)P(x, dy) = d\rho_x(y)$.

In the next statement we collect several properties of the Markov operator
$P$ considered in various spaces.

\begin{proposition} \label{lem P in L^1}
Let  $\sms$, $\nu$, and $\rho$ be as above. Then, for any 
$A\in \Bfin$,

(a) $P(\chi_A) \in L^1(\mu) \ \Longrightarrow \ P(\chi_A) \in L^2(\mu)$;

(b) $P(\chi_A) \in L^1(\mu) \ \Longleftrightarrow \ \dfrac{\rho_x(A)}
{c(x)} \in L^1(\mu) \ \Longrightarrow \ P(\chi_A) \in L^2(\mu)$;

(c) if the function $x \mapsto \int_V\dfrac{d\rho_x(y)}{c(y)} $ is locally
integrable, then $P$ is a densely defined operator in $L^2(\mu)$; 

(d) if $c \in \Lloc$, then 
$$
P(\chi_A) \in L^1(\nu) \cap L^2(\nu);
$$

(e) the measures $\mu$ and $\mu P$ are equivalent if and only if
the function $c^{-1}$ is integrable on $(E_x, \rho_x)$ for $\mu$-a.e.
$x \in V$. The Radon-Nikodym derivative can be found by the formula:
$$
\frac{d(\mu P)}{d\mu}(x)  = \int_V \frac{1}{c(y)} \; d\rho_x(y).
$$ 
\end{proposition}

\begin{proof} (Sketch) 
(a) The result follows from the Schwarz inequality for positive operators,
$$
P(\chi_A)^2 \leq P(\chi_A^2) = P(\chi_A).
$$ 

(b) The criterion for integrability of the function 
$P(\chi_A)$  is proved as follows:
$$
\ba 
\int_V P(\chi_A)(x)\; d\mu(x) = & \iint_{V\times V} \chi_A(y) P(x, dy)\;
d\mu(x) \\
= & \iint_{V\times V} \frac{\chi_A(y)}{c(x)}\; d\rho_x(y)d\mu(x)\\
= &\int_V \frac{\rho_x(A)}{c(x)}\; d\mu(x).
\ea
$$
It follows from (a) that the same computation can be used to show that
 $P(\chi_A)$ is in $L^2(\mu)$ whenever 
 $$
 \dfrac{\rho_x(A)} {c(x)} \in L^1(\mu).
 $$

(c) To prove this result, we refer to the proof of (b) and use  the symmetry
of the measure $\rho$:
$$
P(\chi_A) \in L^2(\mu) \ \Longleftarrow \ P(\chi_A) \in L^1(\mu) 
$$
and
$$
\ba
\int_V P(\chi_A)(x) \; d\mu(x) =  &  \iint_{V\times V} \frac{\chi_A(y)}
{c(x)}\; d\rho_x(y)d\mu(x)\\
= &  \iint_{V\times V} \frac{\chi_A(x)}{c(y)}\; d\rho_x(y)d\mu(x)\\
=& \int_A\left( \int_V \frac{\chi_A(x)}{c(y)}\; d\rho_x(y) \right) 
d\mu(x).\\
\ea
$$
It gives the desired statement. 

(d) Suppose $c(x) \in \Lloc$. Then, using the symmetry of the measure 
$\rho$ and relation (\ref{eq disint formula}), we obtain
$$
\ba 
\int_V P(\chi_A)(x)\; d\nu(x) =& \int_V \left(\int_V \chi_A (y) 
\frac{1}{c(x)} \; d\rho_x(y) \right) \; c(x) d\mu(x)\\
=& \iint_{V\times V}\chi_A(x) \; d\rho_x(y)d\mu(x)\\
=& \int_V \chi_A(x) c(x)\; d\mu(x) \\
= & \int_A c(x)\; d\mu(x) < \infty,
\ea
$$
i.e., $P(\chi_A) \in L^1(\nu)$. The fact that $P(\chi_A) \in L^2(\nu)$ 
is proved as in (a).

(e) The statement will follow from the following chain of equalities:
$$
\ba 
(\mu P)(A) = & \int_V \chi_A \; d(\mu P)\\
= & \int_V P(\chi_A) \; d\mu\\
= & \int_V \left( \int_V \chi_A(y) P(x, dy)\right) d\mu(x)\\
= & \iint_{VtV}  \chi_A(y) \frac{1}{c(x)} \; d\rho_x(y) d\mu(x)\\
= & \int_{V}  \chi_A(x)\left( \int_V \frac{1}{c(y)} \; d\rho_x(y)\right)
 d\mu(x)\\
 = & \int_{A} \left( \int_V \frac{1}{c(y)} \; d\rho_x(y)\right)
 d\mu(x)\\
 = & \int_A \frac{d(\mu P)}{d\mu}(x) \; d\mu(x)
\ea
$$
where 
$$
\frac{d(\mu P)}{d\mu}(x)  = \int_V \frac{1}{c(y)} \; d\rho_x(y).
$$
\end{proof}

Clearly, Proposition \ref{lem P in L^1} can be extended to functions from 
$\Dfin$.

\begin{lemma} \label{lem norm P(chi_A)}
 Let $P$ be a self-adjoint Markov operator in $L^2(\nu)$. 
Suppose that $c \in \Lloc$. Then, for $A \in \Bfin(\mu)$,  
\be\label{eq_nu-norm P^n}
|| P^n(\chi_A) ||^2_{L^2(\nu)} = \rho_{2n}(A\times A), \quad n\in \N,
\ee
where measures $\rho_n$ are defined in (\ref{eq_def rho_n}).
\end{lemma}

\begin{proof}
We recall that if $P$ is a self-adjoint operator in the space $L^2(\nu)$, 
then  $\nu P = \nu$.
Hence,
$$
\ba
 || P^n(\chi_A) ||^2_{L^2(\nu)}  
 = & \langle P^n(\chi_A), P^n(\chi_A)   \rangle_{L^2(\nu)}\\
 = & \langle \chi_A, P^{2n}(\chi_A)   \rangle_{L^2(\nu)}\\
 = & \rho_{2n}(A\times A).
\ea
$$
\end{proof}

\begin{corollary} In conditions of Lemma \ref{lem norm P(chi_A)}, we
have that, for all $n\in \N$,
$$
\int_A c\; d\mu = || \chi_A ||^2_{\h_E(\rho_{2n})} +
 || P^n(\chi_A) ||^2_{L^2(\nu)}. 
$$
\end{corollary}

\begin{proof}
Since $\rho_x^{(n)}(V) = c(x)$ for all $n\in \N$, we can easily
deduce the following equality from Lemma \ref{lem norm P(chi_A)}. 
We use  formula \eqref{eq norm chi_A in H} 
$$
\ba
|| \chi_A ||^2_{\h_E(\rho_n)} = & \rho_n(A \times A^c)\\
 = &\rho_n(A \times V) - \rho_n(A \times A)\\
  = &\int_A c\; d\mu -  \rho_n(A \times A).
\ea
$$ 
\end{proof}

\begin{remark}
It is interesting to compare formula (\ref{eq_nu-norm P^n}) with a similar
result for $|| P^n(\chi_A) ||^2_{\h_E}$ proved in 
\cite{BezuglyiJorgensen2018}, see also \eqref{eq ||chi A||} in Theorem 
\ref{thm_stucture of energy space}.  
$$
\| P^n(\chi_A)\|^2_{\h_E} = \rho_{2n}(A\times A) - \rho_{2n+1}
(A\times A), \qquad n \in \N.
$$
Hence, it follows that
$$
\| P^n(\chi_A)\|^2_{\h_E} =  || P^n(\chi_A) ||^2_{L^2(\nu)}  - 
\rho_{2n+1}(A\times A). 
$$
\end{remark}

\subsection{More on the  embedding operator $J$}  
In this subsection, we return to the study of the operator $J$ defined 
in \eqref{eq embedd J}, see Subsection \ref{subsect J}. We recall that 
the operator $J$ is an isometry if considered acting from 
$L^2(\nu)$ to $L^2(\rho)$, and it is an unbounded operator 
from $L^2(\mu)$ to $L^2(\rho)$. Here we focus on relations between 
$J$ and other operators we study in the paper. 

\begin{lemma}\label{lem P in L^2(rho)}
For any $A \in \Bfin(\mu)$, we have 
$$
|| J(P(\chi_A)) ||^2_{L^2(\rho)} \leq || \chi_A ||^2_{L^2(\nu)}.
$$
\end{lemma}
\begin{proof}
Indeed, we use Schwarz' inequality for  $P$ to show that
$$
\ba 
\iint_{V\times V} J(P(\chi_A))^2(x, y)\; d\rho(x,y) = & 
\int_V P(\chi_A)^2(x)\; d\rho(x, y)\\
\leq  &  \int_V P(\chi_A)(x)\; d\rho(x, y)\\
= &\int_V c(x) P(\chi_A)(x)\; d\mu(x)\\
= & \iint_{\VtV} \chi_A(y) \; d\rho_x(y) d\mu(x)\\
= &  \iint_{\VtV} \chi_A(x) \; d\rho_x(y) d\mu(x)\\
= & \int_A c(x) \; d\mu(x)\\
= &|| \chi_A ||^2_{L^2(\nu)}. 
\ea
$$
\end{proof}

As an illustration of properties of this embedding $J$, we note that the
 function $J(c^{-1})(x, y)$ is not integrable with  respect to $\rho$ but 
 is locally integrable. 

Another useful relation that compares norms of functions is contained  in
the following inequality.
\begin{lemma}
Let $f$ be a function from the finite energy space such that  $f$ and 
$\Delta(f)$ belong to $L^2(\mu)$. Then 
$$
|| Jf ||^2_{L^2(\rho)}  \geq \frac{1}{2}|| f ||^2_{\h_E}.
$$
\end{lemma} 
\begin{proof}
The proof  follows from  \cite[Corollary 7.4]{BezuglyiJorgensen2018}
and Proposition \ref{prop prop of R, P, Delta} (6):
$$
\ba 
\iint_{V\times V} (Jf)^2(x, y)\; d\rho(x, y) = & \iint_{V\times V} f^2(x)\; 
d\rho_x(y) d\mu(x) \\
=& \int_V f^2(x) c(x) \; d\mu(x) \\
\geq  &\frac{1}{2}\langle f, \Delta f\rangle_{L^2(\mu)}\\
= & \frac{1}{2}|| f ||^2_{\h_E}.
\ea
$$
\end{proof}

In the remaining part of this section, we consider the Markov operator $P$
 as an operator acting on functions from the energy space $\h_E$.

\begin{proposition} Assume that $c \in \Lloc$. Then, for every $A\in
 \Bfin(\mu)$, we have
$$
(JP)(\chi_A)(x, y) \in \h_E.
$$
\end{proposition}

\begin{proof}
We need to show that the energy norm of $J(P(\chi_A))$ is finite. 
By Theorem \ref{thm_stucture of energy space},  we find that
$$
\ba 
|| (JP)(\chi_A) ||^2_{L^2(\rho)} =& \frac{1}{2} \iint_{VtV} (P(\chi_A)(x) -
 P(\chi_A)(y))^2\; d\rho(x, y)\\
= & \iint_{VtV} (P(\chi_A)^2(x) - P(\chi_A)(x) P(\chi_A)(y))\; d\rho(x, y).
\ea
$$
To see that the last integral is finite, we first show that $(JP)(\chi_A)$ is 
in $L^2(\rho)$:
$$
\ba 
\iint_{\VtV} P(\chi_A)^2(x) \; d\rho(x, y) \leq & 
\iint_{\VtV} P(\chi_A)(x) \; d\rho_x(y) d\mu(x)\\
= & \int_{V} P(\chi_A)(x)c(x)\; d\mu(x)\\
= & \nu(A)\\
= & \int_A c(x) \; d\mu(x).
\ea
$$
The latter is finite. 

Similarly, one can check that $\iint_{\VtV} P(\chi_A)(x) P(\chi_A)(y)\; d
\rho(x, y)$ is also finite. We leave the proof for the reader.
\end{proof}

Consider a new operator, denoted by $\partial$, which acts from the 
energy space $\h_R$ to $L^2(\rho)$:
\be\label{eq def drop op}  
(\partial f)(x, y) = \frac{1}{\sqrt 2}(f(x) - f(y)),\qquad f \in \h_E 
\ee
Remark that in the theory of electrical networks the analogous transformation
is called a voltage drop operator. 

\begin{lemma}\label{lem d is isom}
The operator $\partial : \h_E \to L^2(\rho)$ defined by 
(\ref{eq def drop op}) is an isometry. 
\end{lemma}

\begin{proof}
The proof is obvious because
$$
|| f ||^2_{\h_E} = \frac{1}{2}\iint_{VtV} (f(x) - f(y))^2 \; d\rho(x,y) =
|| (\partial f) ||^2_{L^2(\rho)}. 
$$
\end{proof}

Since $J : L^2(\nu) \to L^2(\rho)$ is an isometry, then the co-isometry
$J^*$ sends $L^2(\rho)$ to $L^2(\nu)$ according to the 
formula
$$
(J^*g) (x) = \int_V g(x, \cdot) \; P(x, \cdot)
$$
where $g \in L^2(\rho)$. 

In the following proposition, we formulate a relation between operators
$P$, $J^*$, and $\partial$.

\begin{proposition}\label{prop diagram commutes}
The following diagram commutes:
$$
\begin{array}[c]{ccc}
 \mathcal H_E  &\  \stackrel{\wt \Delta}{\longrightarrow} &  \ L^2(\nu)\\
\ \ \ \  \searrow\scriptstyle{\partial} 
&&\nearrow\scriptstyle{J^*}\\
&  L^2(\rho) &
\end{array}
$$
where $\wt \Delta = (\sqrt{2} c)^{-1} \Delta = (\sqrt{2})^{-1} (I - P)$.
\end{proposition}

\begin{proof}
The proof is mainly based on Theorem \ref{prop embedd J} and
the definition of $\partial$. We have
$$
\ba 
(J^*\partial f)(x) = & \frac{1}{\sqrt 2} J^*(f(x) - f(y))\\
= & \frac{1}{\sqrt 2} \int_V (f(x) - f(y)) \; P(x, dy)\\
= & \frac{1}{\sqrt 2} (f(x) - P(f)(x))\\
= & \frac{1}{\sqrt{2}} c(x) \Delta(f)(x).
\ea
$$
\end{proof}

In the next statement, we present several  properties of the operator 
$I - P$.

\begin{corollary}
(1) 
$$
(I - P) \h_E \subset L^2(\nu), 
$$ 

(2) The operator  $I -P$ acting from $\h_E$ to $L^2(\nu)$ is contractive. 

(3) For the operator  $\Delta = c(I - P)$, the following holds
$$
\Delta(\h_E) \subset c L^2(\nu).
$$
\end{corollary}

\begin{proof}

Assertion (1) is a direct consequence of Proposition 
\ref{prop diagram commutes} 
(this result was already mentioned in \cite{BezuglyiJorgensen2018}). 

To see that (2) holds, we recall the formula for the 
norm of a function in the finite energy space $\h_E$:
$$
\| f \|^2_{\h_E} =  \frac{1}{2}\left(\| f - P(f)\|^2_{L^2(\nu)} 
+ \int_V \mathrm{Var}_x (f \circ X_1)\; d\nu \right),
$$
where the meaning of random variables $X_n$ is explained in Section
\ref{subsect Path space MP} below.

(3) is obvious. 
\end{proof}

\section{Transient Markov processes and symmetric measures}
\label{subsect Path space MP}

\subsection{Path-space measure}
We denote by 
$\Omega$ the infinite Cartesian product  $V\times V \times \cdots = 
V^{\N_0}$.  Let $(X_n (\omega): n = 0,1,...)$ be the 
 sequence of random variables  $X_n : \Omega \to V$ such that 
  $X_n(\omega) =  \omega_n$. We call $\Omega$ as 
the path space of the Markov process $(P_n)$. 
 Let  $\Omega_x, x \in V,$ be the set of infinite paths beginning at $x$:
$$
\Omega_x := \{\omega\in \Omega : X_0(\omega) = x\}.
$$
Clearly, $\Omega = \coprod_{x\in V}  \Omega_x$. 

A subset $\{\omega \in
 \Omega : X_0(\omega) \in A_0, ... X_k(\omega) \in A_k\}$ is called 
 a \textit{cylinder set} defined by Borel sets $A_0, A_1, ..., A_k$
 taken from $\B$, $k \in \N_0$.  
 The collection of cylinder sets generates the $\sigma$-algebra 
 $\mathcal C$ of Borel subsets of $\Omega$, and $(\Omega, \mc C)$
 is a standard Borel space. 
Then  the functions $X_n : \Omega \to V$ are Borel. 

On the measurable space $(\Omega, \mathcal C)$, define 
a $\sigma$-finite measure $\lambda$ by 
\be \label{eq def lambda}
\lambda := \int_V \mathbb P_x \; d\nu(x)
\ee
($\la$ is infinite if and only if the measure $\nu$ is infinite).

Denote by 
 $\mathcal{F}_{\leq n}$ the increasing  sequence of $\sigma$-subalgebras 
  such that  $\mathcal{F}_{\leq n}$ is the smallest
 subalgebra for which the functions $X_0, X_1, ... , X_n$ are Borel. By
 $\mathcal F_n$, we denote the $\sigma$-subalgebra $X_n^{-1}(\B)$. 
 Since $X_n^{-1}(\B)$ is a $\sigma$-subalgebra of $\mc C$, there exists 
 a  projection 
 $$
 E_n : L^2(V, \mc C, \la) \to L^2(\Omega, X_n^{-1}(\B), \la).
 $$ 
The projection  $E_n$ is called the \textit{conditional expectation} 
with respect to $X_n^{-1}(\B)$ and satisfies the property:
\be\label{eq cond exp E_n}
E_n(f\circ X_n) = f\circ X_n. 
\ee
 
 Define a probability measure $\mathbb P_x$ on $\Omega_x$. For 
 a cylinder set $(A_1, ... , A_n)$ from $\mathcal F_{\leq n}$  we set
\be\label{eq meas P_x} 
 \mathbb P_x(X_1 \in A_1, ... , X_n \in A_n) 
= \int_{A_{1}}\cdots \int_{A_{n-1}} P(y_{n-1}, A_n) P(y_{n-2},  dy_{n-1})
\cdots  P(x, dy_1).
 \ee
Then $\mathbb P_x$ extends to the Borel sets on $\Omega_x$ by  the
 Kolmogorov extension theorem \cite{Kolmogorov1950}.

The values of $\mathbb P_x$ can be written as 
\be\label{eq meas P_x 2}
 \mathbb P_x(X_1 \in A_1, ... , X_n \in A_n) = 
 P(\chi_{A_1} P(\chi_{A_2}P(\ \cdots\ P(\chi_{A_{n-1}} P(\chi_{A_n})) 
 \cdots )))(x).
\ee
The joint distribution of the random variables $X_i$ is given by 
\be\label{eqjoint distr}
d\mathbb P_x(X_1, ... , X_n)^{-1} = P(x, dy_1) P(y_1, dy_2) \cdots 
P(y_{n-1}, dy_n).
\ee

\begin{lemma}\label{lem st meas space}
The measure space $(\Omega_x, \mathbb P_x)$ is a standard 
probability measure space for $\mu$-a.e. $x\in V$. 
\end{lemma}

We proved in \cite{BezuglyiJorgensen2018} that the Markov process
$P_n$ is irreducible if the initial symmetric measure is irreducible. More
precisely, the statement is as follows.

\begin{theorem}\label{prop from A to B}   
Let $\rho$ be a symmetric measure on $\vv$, and let $A$ and $B$ be 
any two sets from $\Bfin(\mu)$. Then 
 \be\label{eq rho_n vs lambda}
 \rho_n(A \times B)  = \langle \chi_A, P^n(\chi_B)\rangle_{L^2(\nu)} 
 = \lambda (X_0 \in A,  X_n  \in B),\ \  n \in \N.
 \ee
The Markov process $(P_n)$ is irreducible if and only if the 
measure  $\rho$ is irreducible. 
\end{theorem}

 In other words, relation (\ref{eq rho_n vs lambda}) can be interpreted in 
 the following way: 
for the Markov process $(P_n)$, the ``probability''   to get in $B$ for 
$n$ steps starting somewhere in $A$ is exactly  $\rho_n(A \times B) > 0$.

To see that (\ref{eq rho_n vs lambda}) holds, one uses the definition 
of the measure $\lambda$ and formulas \eqref{eq meas P_x} and 
\eqref{eq meas P_x 2}.

\begin{corollary} Let $A_0, A_1, ... , A_n$ be a finite sequence of subsets
 from $\Bfin$. Then 
 $$
\mathbb P_x(X_1 \in A_1, ... , X_n \in A_n)\ |\ x \in A_0) > 0 \ 
\Longleftrightarrow \ \rho(A_{i-1} \times A_i) > 0
$$ 
for $i=1, ... ,n$.
\end{corollary}

It is worth noting that the concept of reversible Markov processes can
be formulated in terms of the measure $\lambda$, roughly speaking
$\lambda$ must be a symmetric distribution. 

\begin{proposition}\label{lem symm distr}
Let the measure $\lambda$ on $\Omega$ be defined by (\ref{eq def 
lambda}). The Markov operator $P$ is reversible if and only if 
$$
\lambda(X_0\in A_0 \ |\ X_1\in A_1) =
\lambda(X_0\in A_1 \ |\ X_1\in A_0). 
$$
\end{proposition}  

\begin{proof} The proof uses the fact that $P$ is reversible if and only is
$P$ is self-adjoint in $L^2(\nu)$. We compute applying 
(\ref{eq meas P_x}):
$$
\ba 
\lambda(X_0\in A_0 \ |\ X_1\in A_1) & = \int_{A_0} \mathbb P_x
(X_1 \in A_1) \; d\nu(x)\\
& = \int_V \chi_{A_0}(x) P(\chi_{A_1} ) (x) \; d\nu(x)\\ 
& = \int_V \chi_{A_1}(x) P(\chi_{A_0} ) (x) \; d\nu(x)\\ 
& = \lambda(X_0\in A_1 \ |\ X_1\in A_0).
\ea
$$
It proves  the statement.
\end{proof}

In the next statement we relate harmonic functions to martingales.
Recall first the definition of a martingale.

Let $(X_n : n \in \N)$ be the Markov chain on $\Omega$ 
with values in $(V, \B)$ defined by $X_n(\omega) = \omega_n$. We recall 
that the space $\Omega$ is represented as the disjoint union of subsets
$\Omega_x := \{ \omega \in \Omega : \omega_0 = x\}$, $x \in V$. 
Let $(\Phi_n : n\in \N_0)$ be a sequence of real-valued random variables
defined on $\Omega$. Then it generates  a  sequence of measurable fields 
of random variables $x \to \Phi_n(x), x \in V,$ defined on the 
corresponding subset $\Omega_x$. Let $\mc C_n$ be the 
$\sigma$-algebra of subsets
of $\Omega$ generated by $\Phi_n^{-1} (B), B \in \B$.
Denote by $\mc C_{\leq n}$ the smallest $\sigma$-subalgebra such
that the functions $\Phi_i, i =1,... n,$ are Borel measurable. These 
$\sigma$-algebras induce $\sigma$-algebras $\mc C_{\leq n}(x)$ on 
every $\Omega_x$.

It is said that the sequence $(\Phi_n)$ is a \textit{martingale} if
$$
\mathbb E_x(\Phi_{n+k}(x)\  |\  \mc C_{\leq n}(x)) = \Phi_n(x), \ \ \ 
 \forall k.
$$
Here $\mathbb E_x$ is the conditional expectation with respect to
the probability  path measure $\mathbb P_x$, see \eqref{eq meas P_x}.

\begin{proposition} Let $P$ be the Markov operator defined by a 
symmetric measure $\rho$. 
For the objects defined above, the following are equivalent:

(i) a Borel function $h$ on $\VB$ is harmonic with respect to the Markov 
operator $P$;

(ii) the sequence $(h \circ X_n : n \in \N_0)$ is a martingale.
\end{proposition}

\begin{proof} It follows from the definition of the Markov chain $(X_n)$,  
path  space  measure  $\mathbb P_x$, and 
\cite[Proposition 2.24]{AlpayJorgensenLewkowicz2018} that,
for any Borel function $f$,
$$
\mathbb E_x(f \circ X_{n+m} \ |\ \mc C_{\leq n}(x)) = 
\mathbb E_x(f \circ X_{n+m} \ |\ \mc C_{n}(x)) = 
P^m(f)\circ X_n.
$$
Hence, we see that  a function $h$ is harmonic if and only if 
$$
\mathbb E_x(h \circ X_{n+m} \ |\ \mc C_{\leq n}(x)) = h\circ X_n,
$$ 
i.e., $(h \circ X_n)$ is a  martingale. 
\end{proof}

\subsection{Green's functions}
In this section, we will work with transient Markov processes.
We first define a Green's function $G(x, A)$. Our main goal is to study 
Green's functions as  elements of the energy space. 

\begin{definition}\label{def G}
Let 
$$
G(x, A) = \sum_{n = 0}^{\infty} P_n(x, A), \qquad A\in \Bfin(\mu), 
x \in V. 
$$
The Markov process is called \textit{transient} if, for every $A \in 
\Bfin$, the function $G(x, A)$ is finite $\mu$-a.e. on $V$. 
\end{definition} 

In this subsection, we will always assume that the Markov process 
$(P_n)$ is transient. 

\begin{lemma} Let $\rho$ be an irreducible symmetric measure. 
Suppose $A\in \Bfin$ be a set such that $G(x, A)$ is finite 
a.e. Then, for  any $B\in \Bfin$, the function $G(x, B)$ is finite for
$\mu$-a.e. $x \in V$.
\end{lemma}

\begin{proof} The proof of this result is straightforward and mainly based
on the definition of irreducible measure, see also Lemma 
\ref{lem-irr measure}. 

\end{proof}

\begin{lemma}\label{lem h-norm for P^n}
 Let $A\in \Bfin$ and let $P$ be a Markov operator defined
by a symmetric measure $\rho$. Then the function $x\mapsto 
P_n(x, A) = P^n(\chi_A)(x)$ belongs to $\h_E$ and
$$
\| P_n( \cdot, A)\|^2_{\h_E} = \rho_{2n}(A\times A) - \rho_{2n+1}
(A\times A), \qquad n \in \N.
$$
\end{lemma}

\begin{proof} 
The proof is based on the facts that $\nu$ is $P$-invariant, $\rho$ is 
symmetric,  and on the definition of the norm in the energy space which
are used in the following computation: 
$$
\ba 
|| P_n(x, A) ||^2_{\h_E} = & \iint_{V\times V} P_n(x, A)(P_n(x, A) -
P_n(y, A)) \; d\rho(x, y)\\
= & \iint_{V\times V} P_n(x, A)(P_n(x, A) - P_n(y, A)) c(x) P(x, dy) \; 
d\mu(x)\\
= & \int_V \left[ P_n(x, A)^2 - P_n(x, A) \int_V P_n(y, A) P(x, dy)\right]
\; d\nu(x)\\
= & \int_V \left[ P_n(x, A)^2 - P_n(x, A) P_{n+1}(x, A) \right]
\; d\nu(x)\\
=& \int_V P_n(x, A)(P_n(x, A)  - P_{n+1}(x, A))\; d\nu(x)\\
=& \int_V \chi_A(x) P^n(P^n(\chi_A)  - P^{n+1}(\chi_A))(x)\; d\nu(x)\\
= & \langle\chi_A(x), P^{2n}(\chi_A)(x) \rangle_{L^2(\nu)} -
  \langle \chi_A(x), P^{2n +1}(\chi_A)(x) \rangle_{L^2(\nu)}  \\
  =& \rho_{2n}(A\times A) -   \rho_{2n +1}(A\times A).
\ea
$$
\end{proof}

\begin{remark}
As a curious observation, we mention that, for any $A \in \Bfin$, 
$$
\rho_{2n}(A\times A) >  \rho_{2n +1}(A\times A).
$$  
It is worth noting that the above formula cannot be extended to direct 
products of sets $A$ and $B$ from $\Bfin(\mu)$. In particular, one can 
prove that  the relation 
$$
\rho_2(A \times B) < \rho(A \times B)
$$
implies that $P(\chi_B - P(\chi_B)) > 0$ a.e. Therefore there would exist a 
harmonic function in $L^2(\nu)$ which is a contradiction. 

\end{remark}


Fix a set $A\in \Bfin$, then we have the family   of measurable
 functions $G_A (x):= G(x, A)$ indexed by sets of finite measure.

\begin{lemma} \label{lem c(I-P)G_A}
For a set $A\in \Bfin$, the equality
$$
c(x)(I - P)(G_A)(x) = c(x)\chi_A(x)
$$ 
 holds. Equivalently, 
$$
\Delta G_A(x) = c(x)\chi_A(x).
$$
\end{lemma}

\begin{proof} We compute using the definition of Green's function and the
fact that the series $\sum_n P_n(x, A)$ is convergent for all $x$ and 
all $A\in \Bfin(\mu)$:
$$
\ba
c(x)(I - P)G_A(x) = \ & c(x)(I - P)\sum_{n=0}^\infty  P_n(x, A) \\
= \ & c(x) \sum_{n=0}^\infty  P_n(x, A) - c(x) 
\sum_{n=1}^\infty  P_n(x, A) \\
= \ & c(x) \chi_A(x).
\ea
$$
\end{proof}

 \begin{theorem}\label{prop on G_A} 
 For the objects defined above, we have the following 
 properties. 
 
  (1) For any sets $A, B \in \Bfin$, we have
\be\label{eq_inner prod G_A and G_B}
\langle G_A, G_B\rangle_{\h_E} = \sum_{n=0}^\infty \rho_n(A\times B);
\ee  
  and, in particular, 
  \be\label{eq norm G_A}
  \| G_A(x) \|_{\h_E}^2 = \sum_{n=1}^\infty \rho_n(A\times A).
  \ee 
 
 (2) For any $f \in \h_E$ and $A\in \Bfin(\mu)$, 
 $$
 \langle f, G_A\rangle_{\h_E} = \int_A f\; d\nu.
 $$
 Furthermore, if 
 \be\label{eq def mc G}
 \mc G := \mathrm{span} \{G_A(\cdot) : A \in \Bfin\},
 \ee
 then $\mc G$ is dense in the energy space $\h_E$. 

 \end{theorem} 
 
 \begin{proof} (1) We prove \eqref{eq norm G_A} here.
 Relation \eqref{eq_inner prod G_A and G_B} is proved similarly.
One has
 $$
 \ba 
\| G_A(x) \|_{\h_E}^2 & = \iint_{\VtV}(G_A(x) - P_A(y))^2\; d\rho(x,y)\\
 & = \iint_{\VtV} G_A(x)  (G_A(x) - G_A(y)) \; d\rho(x,y)\\
  & = \iint_{\VtV} G_A(x)  (G_A(x) - P_A(y)) c(x) P(x, dy) d\mu(x))\\
  &= \int_{V} G_A(x)  [G_A(x) -  P(G_A)(x)] c(x) \; d\mu(x))\\
 &= \int_{V} G_A(x)  [\sum_{n=0}^\infty P^n(\chi_A)(x) -  
 \sum_{n=0}^\infty P^{n+1}(\chi_A)(x)   c(x) \; d\mu(x))\\
& = \int_V  \sum_{n=0}^\infty P^n(\chi_A)(x) \chi_A(x) \; d\nu(x)\\
& = \sum_{n=0}^\infty \langle\chi_A, P^n(\chi_A  \rangle_{L^2(\nu)}\\
& = \sum_{n=0}^\infty \rho_n(A\times A).  
 \ea
 $$
 
 For (2), 
$$
\ba 
\langle f, G_A\rangle_{\h_E} & = \frac{1}{2} \iint_{\VtV} 
(f(x) - f(y))(G_A(x) - G_A(y))\; d\rho(x,y)\\
&= \iint_{\VtV} (f(x) G_A(x) -  f(x)G_A(y))\; d\rho(x, y)\\
& = \int_V  \left[f(x) G_A(x) c(x)  - f(x) \left(\int_V G_A(y) P(x, dy)
\right)c(x)\right]\; d\mu(x)\\
&= \int_V  f(x) c(x) \left[\sum_{n=o}^\infty  P^n(\chi_A)(x) -
\sum_{n=o}^\infty  P^{n+1}(\chi_A)(x)\right] \; d\mu(x)\\
&= \int_V  f(x) \chi_A(x) c(x) \; d\mu(x)\\
& = \int_A f\; d\nu.
\ea
$$ 
It follows from the proved relation that if $\langle f, G_A\rangle_{\h_E}
 =0$ for all $A\in \Bfin(\mu)$, then  $f=0$, and $\mc G$ is dense in 
 $\h_E$.
 \end{proof}

Let $\Dfin(\mu) \subset L^2(\mu)$ denote, as usual,  the space spanned by
characteristic functions, and  let 
$\mc G$ be as in (\ref{eq def mc G}).  Then the following two operators,
$J$ and $K$, are densely defined 
\be\label{eqdef J and K}
J: \chi_A \mapsto \chi_A : \Dfin \to \h_E, \qquad 
K : G_A \mapsto c(I - P)(G_A) : \mc G \to L^2(\mu)
\ee
where $A\in \Bfin(\mu)$.

\begin{proposition}\label{prop J K symm pair}
The operators $J$ and $K$ form a symmetric pair, i.e., 
\be\label{eq_symm pair}
\langle J\va, f\rangle_{\h_E}  = \langle \va, K(f)\rangle_{L^2(\mu)} 
\ee
where $\va\in \Dfin$ and $f \in \mc G$.
\end{proposition}

\begin{proof}
To prove \eqref{eq_symm pair} it suffices to check that it holds for 
 $\va = \chi_A$ and $f = G_B$ where $A, B \in \Bfin(\mu)$. 
 For these functions, we show that the both inner products are equal to
 $\nu(A\cap B)$. 
 
By Lemma \ref{lem c(I-P)G_A}, we have
$$
\ba 
\langle \chi_A, K(G_B)\rangle_{L^2(\mu)}  & = \langle \chi_A, 
c \chi_B\rangle_{L^2(\mu)} \\
& = \int_V \chi_A c \chi_B  \; d\mu\\
& = \nu(A \cap B).
\ea
$$
On the other hand, for the same functions $\va$ and $f$, we compute 
the inner product in the finite energy Hilbert space using the symmetry 
of $\rho$: 
$$
\ba 
\langle J(\chi_A), G_B \rangle_{\h_E} & = \frac{1}{2} \iint_{\VtV}
(\chi_A(x) - \chi_A(y)) (G_B(x) -  G_B(y)) \; d\rho(x,y)\\
& = \iint_{\VtV} (\chi_A(x)  G_B (x) - \chi_A(x)  G_B (y)) \;
d\rho(x, y)\\
& = \iint_{\VtV} [\chi_A(x) \sum_{n=0}^\infty P^n(\chi_B)(x) \\
& \ \ \ \ \ \ - \chi_A(x) \sum_{n=0}^\infty P^n(\chi_B)(y) ]c(x) 
P(x, dy)d\mu(x)\\ 
& = \int_{V} [\chi_A(x) \sum_{n=0}^\infty P^n(\chi_B)(x) \\
& \ \ \ \ \ \ - \chi_A(x) \sum_{n=0}^\infty \int_V P^n(\chi_B)(y)P(x,dy)] 
c(x) d\mu(x)\\ 
& = \int_{V} [\chi_A(x) \sum_{n=0}^\infty P^n(\chi_B)(x) 
- \chi_A(x) \sum_{n=1}^\infty P^n(\chi_B)] \; d\nu(x)\\
& = \int_V \chi_A(x) \chi_B(x) \; d\nu(x)\\
& = \nu(A\cap B).
\ea
$$
 
\end{proof}

\begin{corollary} The finite energy Hilbert space admits the orthogonal 
decomposition
$$
\h_E = \overline{J(\Dfin(\mu))} \oplus \mc Harm.
$$
In particular, for every $B\in \Bfin(\mu)$, we have $G_B = G_1 \oplus 
G_2$, where $G_1\in \overline{J(\Dfin(\mu))}$ is always non-zero. 
\end{corollary}

\begin{proof}
Indeed, if one assumed that $G_1 = 0$, then we would have that 
$G_B$ is orthogonal to $\overline{J(\Dfin(\mu))}$. This contradicts
Theorem \ref{prop on G_A}. 
\end{proof}

We conclude this section with the following result that was proved in
\cite{BezuglyiJorgensen2018}:

\begin{theorem}
Let $(P_n)$ be a transient Markov process, and let $G(x, A)$ be the 
corresponding Green's function. Then, for any $f \in \h_E$, we have
the decomposition
$$
f = G(\va) + h
$$
where $h$ is a harmonic function and $\va \in L^2(\nu)$. 
\end{theorem}

\section{Discretization of the graph $\Bfin(\mu)$}
\label{sect discretization}
Let $(V, \B, \mu)$ be a $\sigma$-finite measure space, and let $\rho$ be
a symmetric measure on $\vv$.  We will  associate with 
$(V, \B, \mu)$ and $\rho$ a sequence of countably infinite graphs 
$\mc G_n$ equipped with conductance functions $c_n$ such that the 
weighted graphs $(\mc G_n, c_n)$ can be viewed as a discretization of
of the uncountable graph $\Bfin$ considered in
 \cite{BezuglyiJorgensen2018}.  

We first recall a few facts from \cite{BezuglyiJorgensen2018}. 

\begin{lemma}\label{lem A x A^c finite} Suppose that  $c(x) \in
 L^1_{\mathrm{loc}}(\mu)$. Then, for any set $A\in \Bfin$, 
\be\label{eq finite rho on Bfin}
 \rho(A \times A^c) < \infty
 \ee
 where $A^c = V \setminus A$. The converse is not true, in general. 
\end{lemma}

We can view at the set $\Bfin = \Bfin(\mu)$ as an uncountable  graph 
$\mathcal G$ whose vertices are sets $A$ from $\Bfin$ and edges are defined as follows.
 For a symmetric measure $\rho$ defined on 
$(V \times V, \B\times B)$, we say that two sets $A$ and $B$ from
$\Bfin$ are connected by an edge $e$  if $\rho(A \times B) > 0$.  

This definition is extended to get finite paths in the graph $\mc G$. It is
said that there exists a finite path  in the graph $\mc G$ from $A$ to $B$
 if there exists a sequence $\{A_i  : i = 0, ..., n\}$ of sets  from $ \Bfin$ 
(vertices of  $\mc G$) such that $A_0 = A, A_n = B$ and 
$\rho(A_i \times A_{i+1}) > 0,  i = 0, ... n-1$. 

\begin{theorem}\label{prop connectedness} Let $(V, \B, \mu)$ be as 
above,  and let $\rho$ be a symmetric irreducible measure on 
$\vv$.  Then
 any two sets $A$ and $B$ from the graph $\mc G$ are connected by a finite  path, i.e., the graph $\mc G$ is connected. 
\end{theorem}

\begin{proof}
We will show that there exists a finite sequence $(A_i : 0 \leq i \leq n)$ of
 disjoint subsets from $\Bfin$ such
 that $A_0 = A$,  $\rho(A_i \times A_{i+1}) > 0$, and $\rho(A_n \times
B) >0$, $i = 0, ..., n-1$.

 If $\rho(A \times B) > 0$, then nothing to prove, so that we can
assume that $\rho(A \times B) = 0$.

 Let $\xi =(C_i : i \in \N)$ be a partition of $V$ into disjoint subsets of
positive finite measure such that
$C_i \in \Bfin$ for all $i$. Without loss of generality, we can assume that 
the sets $A$ and $B$ are included in $\xi$. Let for definiteness, 
$A = C_0$. 

Since $\rho(A \times A^c) > 0$ (by Lemma \ref{lem A x A^c finite}),
there exists a set $C_{i_1}\in \xi$ such that
$\rho(A \times C_{i_1}) > 0$ and $\rho(A \times C_{j}) = 0$ for
all $0 <j < i_1$. Set
$$
A_1 := \bigcup_{0 < j  \leq i_1} C_j.
$$
It is clear that $A_1 \in \Bfin$ and $\rho(A_0 \times A_1) > 0$.
If $\rho(A_1 \times B) > 0$, then we are done. If not, we proceed as
follows.
Because of the property $\rho(A_1 \times A_1^c)> 0$,
there exists some $i_2 > i_1$ such that $\rho(A_1 \times C_{i_2})
 > 0$ and $\rho(A_1 \times C_{j}) = 0$ for
all $i_1 <j < i_2$. Set
$$
A_2 := \bigcup_{ i_1 < j \leq i_2} C_j.
$$
Then $\rho(A_1 \times A_2)> 0$, and we check whether $\rho(A_2 \times
 B) > 0$. If not, we continue in the same manner by constructing
  consequently disjoint sets $A_i$ satisfying the property 
  $\rho(A_i \times A_{i+1}) > 0$.
Since $B$ is an element of $\xi$, this process will terminate. This means
that  there exists some
$n$ such that $A_n \supset B$. This argument proves the proposition.

\end{proof}

Given a $\sigma$-finite measure space $\sms$, consider  a sequence of 
measurable partition $\{\xi_n\}_{n\in \N}$ such that

(i) $\xi_n = (A_n(i) : i \in \N), \ \bigsqcup_i A_n(i) = V, \ A_n(i) \in 
\Bfin(\mu)$;

(ii)  $\xi_{n+1}$ refines $\xi_n$, i.e., every element $A_n(i)$ of the 
partition $\xi_n$ is the union of some elements of $\xi_{n+1}$:
$A_n(i) = \bigcup_{j \in \Lambda_n(i)} A_{n+1}(j)$ where 
$\Lambda_n(i)$ is a finite of $\N$;

(iii) the set $\{A_n(i) : i \in \N, n\in \N\}$ generates the Borel
$\sigma$-algebra $\B$.

If for every $i$, the cardinality of the set $\Lambda_i$ is bigger than one,
we say that $\xi_{n+1} $ refines $\xi_n$ strictly.

It is well known, see e.g. \cite{Kechris1995}, that, for any point $x\in V$,
there exists a sequence $i_n(x)$ such that $A_{n+1}(i_{n+1}(x))
 \subset  A_n((i_n)(x))$ and 
\be\label{eq_A_n shrinks to x} 
\{x\} = \bigcap_{n \in \N} A_n(i_n(x))
\ee

Suppose $\rho$ is a symmetric measure on $\vv$. We define a sequence
 of non-negative Borel functions $c^{(n)}$ on $\vv$ by setting 
$$
c^{(n)}_{xy} := \rho(A_n(i_n(x)) \times A_n(i_n(y)))
$$
 for any $x, y$ from $V$. Clearly, $c^{(n)}_{xy}$ is a piecewise constant
 function.
 
 \begin{lemma}\label{lem decreasing c^n} 
 For a given sequence of strictly refining partitions 
 $(\xi_n)$, the sequence of functions $(c^{(n)}_{xy})$ is monotone 
 decreasing. 
 \end{lemma}
 
 \begin{proof}
 The proof is straightforward. For $x, y \in V$, let the sequences 
 $(A_n(i_n(x)))$ and  $(A_n(j_n(y)))$ shrink to the points $x$ and $y$,
 respectively, according to \eqref{eq_A_n shrinks to x}. By assumption
  of the lemma, $A_{n+1}(i_{n+1}(x)) $ is a proper subset of
  $A_n(i_n(x))$. Hence,
  $$
  \ba 
  c_{xy}^{(n+1)} & = \rho(A_{n+1}(i_{n+1}(x))  \times 
  A_{n+1}(j_{n+1}(y)) \\
  & < \rho(A_{n}(i_{n}(x))  \times A_{n}(j_{n}(y))) \\
&  = c^{(n)}_{xy}.
  \ea
  $$
 \end{proof}
 
 We now can define a sequence of discrete graphs (weighted networks)
 $G_n =  (V_n, E_n, w_n) $. The vertex set
 $V_n$ is formed by the atoms of the partition $\xi_n$, i.e., by the
 sets $\{A_n(i) : i \in \N_0\}$; therefore $V_n$ can be identified
 with $\N_0$. The set of edges $E_n$ consists of pairs $(i, j)$ such 
 that 
 $$
(i, j) \in E_n  \ \Longleftrightarrow \ \rho(A_{n}(i)  \times A_{n}(j)) > 0.
 $$
The weight function is $w_n(i, j) = \rho(A_{n}(i)  \times A_{n}(j))$. 
 
\begin{lemma}
Let $\rho$ be a symmetric irreducible measures on $\vv$. Then the
weighted graph $G_n$ is connected for every $n$.   
\end{lemma}
 
 It follows from Lemma \ref{lem decreasing c^n} that 
 $$
 c_{xy} = \lim_{n\to \infty} c^{(n)}_{xy}
 $$
exists and is a Borel positive function. Since the measure $\rho$ is 
symmetric, we conclude that $c_{xy} = c_{yx}$.

Next, we define 
$$
c^{(n)}(x) = \sum_{j} \rho(A_{n}(i_{n}(x))  \times A_{n}(j)) 
= \sum_{y \sim_n x} c^{(n)}_{xy}
$$ 
 where $x \sim_n y $ if and only if $c^{(n)}_{xy} > 0$. It can be seen
 that
 \be\label{eq_c^n via V}
 c^{(n)}(x) = \rho(A_{n}(i_{n}(x))  \times V). 
 \ee

\begin{lemma}
The sequence $(c^{(n)}(x))$ is monotone decreasing for every $x \in V$
and 
$$
c(x) := \lim_{n\to \infty} c^{(n)}(x) = \rho_x(V). 
$$
\end{lemma} 

\begin{proof}
Indeed, we see from \eqref{eq_c^n via V} that 
$$
c^{(n+1)}(x) = \rho(A_{n+1}(i_{n+1}(x))  \times V) 
< \rho(A_{n}(i_{n}(x))  \times V) =  c^{(n)}(x).
$$
Hence, the Borel function $c(x)$ is well defined for every $x$. 
Because $\bigcap_n A_{n}(i_{n}(x)) =\{x\}$, we obtain that 
$c(x) = \rho_x(V)$. 
\end{proof}

\textbf{Acknowledgments.} The authors are pleased to thank colleagues and
 collaborators, especially members of the seminars in Mathematical Physics 
 and Operator Theory at the University of Iowa, where versions of this work 
 have been presented. We acknowledge very helpful conversations with 
 among others Professors Paul Muhly, Wayne Polyzou; and conversations at 
 distance with Professors Daniel Alpay, and his colleagues at both Ben Gurion 
 University, and Chapman University. 

\bibliographystyle{alpha}
\bibliography{references2Markov}

\end{document}